\let\SIAMoriglabel\label
\let\SIAMorigrefstepcounter\refstepcounter
\AddToHook{package/hyperref/before}{%
  \let\label\SIAMoriglabel
  \let\refstepcounter\SIAMorigrefstepcounter}
\documentclass[final,onefignum,onetabnum]{siamart220329}

\usepackage{amsfonts}
\usepackage{amssymb}
\usepackage{amsopn}
\usepackage{booktabs}
\usepackage{graphicx}
\usepackage{mathtools}
\usepackage{mathrsfs}
\usepackage{microtype}
\usepackage{placeins}

\newcommand{\T}{\mathbb T}
\newcommand{\Z}{\mathbb Z}
\newcommand{\E}{\mathbb E}
\newcommand{\cF}{\mathcal F}
\newcommand{\cH}{\mathcal H}
\newcommand{\cV}{\mathcal V}
\newcommand{\Risk}{R}
\newcommand{\Iell}{\mathsf I_{\ell}}
\newcommand{\Irho}{\mathsf I_{\rho}}
\newcommand{\Gell}{\mathsf G_{\ell}}
\newcommand{\Grho}{\mathsf G_{\rho}}
\newcommand{\Sell}{\mathsf S_{\ell}}
\newcommand{\Srho}{\mathsf S_{\rho}}
\newcommand{\Phior}{\boldsymbol\Phi}
\newcommand{\norm}[1]{\lVert #1\rVert}
\newcommand{\bignorm}[1]{\bigl\lVert #1\bigr\rVert}
\newcommand{\ip}[2]{\left\langle #1,#2\right\rangle}
\newcommand{\dd}{\,\mathrm d}

\newcommand{\ft}{\mathrm{ft}}
\DeclareMathOperator{\dist}{dist}

\DeclareMathOperator{\tr}{tr}

\newsiamremark{remark}{Remark}
\newsiamremark{example}{Example}

\title{Non-asymptotic Analysis of Mat\'ern Regression:\\
The Roles of Target and Kernel Lengthscales}
\author{Daniel Sanz-Alonso\thanks{\raggedright Department of Statistics,
University of Chicago, Chicago, IL 60637 USA.\\
\email{sanzalonso@uchicago.edu}.}}
\headers{Non-asymptotic Analysis of Mat\'ern Regression}{D. Sanz-Alonso}
\makeatletter
\let\combinedsavedmaketitle\maketitle
\let\combinedsavedatmaketitle\@maketitle
\let\combinedsavedthanks\thanks
\makeatother

\ifpdf
\hypersetup{
  pdftitle={Non-asymptotic Analysis of Matern Regression:
  The Roles of Target and Kernel Lengthscales},
  pdfauthor={Daniel Sanz-Alonso}
}
\fi

\begin{document}

\maketitle

\begin{abstract}
Theoretical guarantees for kernel regression are typically formulated in terms
of smoothness, but practical accuracy depends critically on how the design
resolution compares with the target and kernel lengthscales.  We develop a finite-sample theory for
Mat\'ern regression on periodic domains with quasi-uniform designs, covering
noiseless interpolation and noisy kernel ridge regression.  Our minimax result
shows that accurate recovery requires a design dense enough to resolve the
target lengthscale and sufficient information at that scale to overcome noise.
We prove that Mat\'ern interpolation additionally requires the design to resolve the kernel
lengthscale: if the kernel is too short relative to the point spacing, an
additional error remains even when the target is well resolved.  For noisy
Mat\'ern regression, we derive a three-term fixed-ridge risk characterization---target-scale bias,
kernel-scale bias, and variance---and show that optimizing over the ridge
parameter yields four distinct contributions.  When the target
is no more than twice as smooth as the kernel, choosing a kernel lengthscale
longer than the target does not worsen the oracle risk, whereas choosing one too
short can.  Thus these resolution conditions determine when accurate recovery
becomes possible, while
smoothness determines how rapidly the error decreases thereafter.
\end{abstract}



\section{Introduction}\label{sec:introduction}

In approximation theory and nonparametric regression, results are commonly
formulated in terms of smoothness.  A typical theorem assumes that a target
function belongs to a Sobolev ball and then describes an algebraic convergence
rate as the design becomes dense.  Such a
statement identifies the eventual decay rate but does not answer a key
finite-sample question: does the design resolve the spatial scale on which the
target varies?  A smooth target with a short lengthscale can be much harder to
learn than a rougher target with a longer lengthscale before either problem
reaches its asymptotic regime.

This paper gives a scale-explicit answer.  Throughout, we use Mat\'ern
regression in a broad sense that includes noiseless minimum-norm interpolation
when the ridge parameter is zero.  We work on the \(d\)-dimensional torus
with an arbitrary prescribed quasi-uniform design of \(N\) points.  The target
class has radius \(A\), lengthscale \(\ell\), and
Sobolev regularity \(\beta\); the Mat\'ern kernel has an independently chosen
lengthscale \(\rho\) and regularity \(\tau\).  We write \(h\) for the fill
distance and \(\sigma\) for the noise standard deviation.  Our estimates track
the distinct roles of these quantities and remain uniform as the target and
kernel lengthscales, design resolution, signal amplitude, and noise level vary
with the sample size.

The main conclusions can be stated informally as follows.  The minimax theorem
identifies two requirements imposed by the target.  Uniform recovery is
impossible unless the design resolves its lengthscale, \(h\lesssim\ell\), and
the information available at that scale satisfies
\[
  \Iell:=\frac{A^2N\ell^d}{\sigma^2}\gtrsim1.
\]
For a quasi-uniform design, \(N\ell^d\) counts observations in a target-scale
volume up to constant factors, and \(\Iell\) multiplies this count by
\(A^2/\sigma^2\).

Mat\'ern interpolation imposes an additional requirement determined by the
kernel: the design must resolve the kernel lengthscale; that is,
\(h\lesssim\rho\).  In noisy Mat\'ern kernel ridge regression (KRR), the kernel scale
also has its own information level,
\[
  \Irho:=\frac{A^2N\rho^d}{\sigma^2}.
\]
The resulting oracle risk is characterized by the largest of four
contributions: geometry and information at each of the target and kernel
scales.  The scale ratios determine when the error begins to decrease, while
smoothness determines the decay exponents thereafter.

Our analysis further reveals that kernel lengthscales above and below the
target lengthscale affect the best achievable risk differently.  In the range
\(d/2<\beta\leq2\tau\), optimizing the ridge parameter
separately for each kernel lengthscale shows that taking \(\rho\geq\ell\)
does not degrade the best statistical risk, although it can worsen numerical
conditioning.  Taking \(\rho<\ell\) is not immediately harmful, but a sufficiently small kernel
lengthscale introduces an additional limitation from point spacing or noise.
This distinction is invisible in a fixed-lengthscale asymptotic statement,
where all lengthscale dependence is absorbed into constants.

\subsection{Contributions and proof ideas}

Our three main results are stated in \cref{sec:main-results}.
\Cref{thm:minimax} is an estimator-independent minimax benchmark.
\Cref{thm:interpolation} is a sharp Mat\'ern interpolation theorem that holds for every
bounded-mesh-ratio design.  \Cref{thm:oracle} is a finite-sample oracle
characterization for scalar-ridge Mat\'ern KRR.  Uniformly over all admissible
scales, and writing \(\Risk_\rho^{\mathrm{or}}\) for the oracle risk defined
below, the oracle risk has the compact form
\begin{equation}\label{eq:intro-four-term}
 \frac{\Risk_\rho^{\mathrm{or}}}{A^2}
 \asymp
 1\wedge\max\left\{
 (h/\ell)^{2\beta},
 (h/\rho)^{4\tau},
 \Iell^{-2\beta/(2\beta+d)},
 \Irho^{-4\tau/(4\tau+d)}
 \right\}.
\end{equation}
The individual decay exponents in \eqref{eq:intro-four-term} have classical
and recent precedents.  The contribution is their scale-uniform coupling,
together with matching kernel-specific lower bounds for every prescribed
quasi-uniform design.  Under the fixed-ridge resolution condition stated in
\cref{thm:oracle}, we prove the sharper pointwise statement that, for every
deterministic scalar ridge parameter, the risk is characterized by three
terms: target bias, kernel bias, and variance;
\eqref{eq:intro-four-term} follows by optimizing this profile and using the
zero estimator when the risk remains of order \(A^2\).

The interpolation and KRR upper bounds first truncate the target
at the effective resolution and then use scale-uniform sampling stability.  The
retained band-limited part is written as the Mat\'ern operator applied to a
source.  Although the norm of this source grows with the kernel lengthscale,
the recovery resolvent supplies exactly the compensating decay.  The lower bound for a
small kernel lengthscale uses a separated-point Fourier frame inequality.  It
shows that fitting even a low Fourier mode with Mat\'ern translates having a
small kernel lengthscale necessarily leaks energy into unresolved modes.  A rank-one
resolvent calculation propagates the same obstruction to positive ridge.
Finally, let \(r\) be the larger of the design resolution and the smoothing scale
imposed by ridge regularization, as defined in \eqref{eq:eta-r}.  A
low-frequency argument shows that KRR retains on the order of \(r^{-d}\)
effective degrees of freedom.  Since noise contributes variance of order
\(\sigma^2/N\) per degree of freedom, this yields the matching variance lower
bound \(\sigma^2/(Nr^d)\).

\subsection{Relation to prior work}

Sampling numbers, optimal recovery, and minimax estimation provide many
ingredients for \cref{thm:minimax}; see
\cite{kriegullrich2021,krieg2023sampling,dolbeault2023sampling,
devore2025optimal,grochenig2020sampling}.  The
classical literature on radial basis functions and Sobolev kernels establishes
fill-distance error bounds, inverse estimates, and saturation phenomena
\cite{schaback1999improved,schabackwendland2002inverse,
narcowich2006sobolev,wendland2005scattered,bejancu2022matern}.
Recent work studies smoothness misspecification, shape parameters, and KRR
saturation from several complementary viewpoints
\cite{teckentrup2020convergence,sanzalonso2025misspecification,
tuowang2020kriging,wynne2021misspecified,wangjing2022krr,
fischersteinwart2020sobolev,lizhanglin2024saturation,
larssonschaback2024scaling,
wenzelsantin2026shape,wenzel2026inverse,
karvonensantinwenzel2025superconvergence}.

In particular, Wenzel and Santin \cite{wenzelsantin2026shape} obtain sharp
shape-parameter direct, inverse, and saturation statements along asymptotic
quasi-uniform sequences.  Their results do not give a class-wise finite-\(h\)
lower bound for each prescribed design and do not treat noise or minimax
regression.  Recent inverse and superconvergence results further characterize
the approximation rates associated with power-space smoothness
\cite{wenzel2026inverse,karvonensantinwenzel2025superconvergence}.  Refined multivariate
Ingham inequalities also give stability estimates for finitely smooth kernel
matrices \cite{wenzeliske2026spectral}.  Our new lower-bound mechanism is the
finite-exclusion frame combined with a rank-one recovery resolvent: it gives a
class-wise obstruction at fixed \(h\), uniformly over every prescribed
bounded-mesh-ratio design, and connects that obstruction to the target-scale
terms in the minimax risk for noisy regression.

A complementary lengthscale-aware perspective is developed by Addy, Latz, and
Teckentrup \cite{addylatzteckentrup2026}.  They construct lengthscale-informed
sparse grids for high-dimensional anisotropic interpolation with separable
Mat\'ern kernels, adapting both the grid and kernel to the target anisotropy,
and establish upper bounds and fast algorithms.  In contrast, we work with
arbitrary quasi-uniform designs, and in our setting the target and kernel
lengthscales may differ.  We also allow noisy observations and ridge
regularization and prove
matching lower and upper bounds for minimax recovery, interpolation, and the
scalar-KRR oracle.  Thus their work uses lengthscale information to design
efficient high-dimensional interpolation schemes, whereas ours identifies
unavoidable resolution and information barriers for a prescribed design.

Gaussian process (GP) contraction theory for rescaled Mat\'ern priors captures
local sample size and penalizes very small kernel lengthscales
\cite{vandervaartvanzanten2009adaptive,fangbhadra2025rescaled}, but uses
design-point loss and has no geometric \(h/\rho\) barrier.

The deterministic exponent \(4\tau\) and its statistical bias--variance
counterpart \(4\tau/(4\tau+d)\) have fixed-scale predecessors
\cite{wenzelsantin2026shape,wenzel2026inverse,lizhanglin2024saturation}.  Our
key contribution is a single finite-sample characterization coupling the target and
kernel lengthscales with design resolution, regularization, and noise,
uniformly over every sufficiently fine prescribed quasi-uniform design.  In the regime
\(d/2<\beta\leq2\tau\), the resulting phase structure reveals a qualitative asymmetry not visible
in the separate fixed-scale results: after optimizing ridge, \(\rho\geq\ell\)
does not worsen statistical risk, whereas a sufficiently small ratio
\(\rho/\ell\) introduces additional barriers from point spacing or noise.  In
particular, the additional geometric barrier is not a lattice-specific aliasing
phenomenon; it holds uniformly for every prescribed quasi-uniform design.

\subsection{Organization}

\Cref{sec:setting} fixes the model and assumptions, and
\cref{sec:main-results} states the three main theorems.  \Cref{sec:proofs}
proves them after collecting preliminary estimates.  \Cref{sec:numerics}
reports numerical experiments, and \cref{sec:discussion} discusses scope and
limitations.  We provide auxiliary results and technical details in the supplement.

\section{Model and assumptions}\label{sec:setting}

\subsection{Observations and Mat\'ern regression}

Let \(\T^d=(\mathbb R/\mathbb Z)^d\) have normalized Haar measure.  At
distinct design points \(X_N=\{x_1,\ldots,x_N\}\subset\T^d\), we observe
\begin{equation}\label{eq:data-model}
 y_i=f(x_i)+\xi_i,
 \qquad
 \xi_i\stackrel{\mathrm{iid}}{\sim}\mathcal N(0,\sigma^2),
 \qquad 1\leq i\leq N.
\end{equation}
The estimation problem is to recover the unknown real-valued target function \(f\)
under squared \(L^2(\T^d)\) loss.  The case \(\sigma=0\) is noiseless
scattered-data approximation.

For \(\tau>d/2\) and \(0<\rho\leq1\), consider the periodic Mat\'ern kernel
\(K_{\rho,\tau}\) with Fourier coefficients
\begin{equation}\label{eq:matern-spectrum}
 \mu_{\rho,k}
 =\rho^d\bigl(1+(2\pi\rho|k|)^2\bigr)^{-\tau},
 \qquad k\in\Z^d.
\end{equation}
The factor \(\rho^d\) keeps the kernel diagonal
\(K_{\rho,\tau}(0)\) of order one uniformly for \(0<\rho\leq1\), so varying
\(\rho\) does not also rescale the kernel amplitude.  It also makes
\(\eta=\lambda/\rho^d\) the natural normalized ridge parameter below.
The kernel lengthscale is \(\rho\); equivalently, \(\rho^{-1}\) is the shape
parameter in the standard radial-basis-function scaling convention
\cite{larssonschaback2024scaling}.  The exponent \(\tau\) is the Sobolev
order of the associated reproducing kernel Hilbert space (RKHS); the usual
Mat\'ern smoothness parameter is \(\tau-d/2\); see
\cite{stein1999interpolation}.

Let \(\cH_{\rho,\tau}\) be the RKHS of \(K_{\rho,\tau}\).  For
\(\lambda\geq0\), define
\begin{equation}\label{eq:krr}
 \widehat f_{\lambda,\rho}
 \in\arg\min_{g\in\cH_{\rho,\tau}}
 \left\{
 \frac1N\sum_{i=1}^N|y_i-g(x_i)|^2
 +\lambda\norm{g}_{\cH_{\rho,\tau}}^2
 \right\}.
\end{equation}
At \(\lambda=0\), the objective may have multiple minimizers; we select the
minimum-RKHS-norm interpolant of the data.  For \(\sigma=0\), this is ordinary
Mat\'ern interpolation.

\subsection{Target class and design geometry}

Write \(e_k(x)=\exp(2\pi i k\cdot x)\) and
\(f_k^{\ft}:=\int_{\T^d}f(x)e^{-2\pi i k\cdot x}\dd x\) for the Fourier
coefficients of \(f\).  For \(\beta>d/2\), \(A>0\), and \(0<\ell\leq1\),
define
\begin{equation}\label{eq:target-class}
 \norm{f}_{\beta,\ell}^2
 :=\sum_{k\in\Z^d}
 \bigl(1+(2\pi\ell|k|)^2\bigr)^\beta |f_k^{\ft}|^2,
 \qquad
 \cF_\beta(A,\ell):=\{f:\norm{f}_{\beta,\ell}\leq A\}.
\end{equation}
For fixed \(\ell>0\), \(\norm{\cdot}_{\beta,\ell}\) is equivalent to the
periodic \(H^\beta\) norm, with \(\ell\)-dependent constants; when \(\ell=1\),
the two norms agree exactly.  Its weight grows
above frequency \(\ell^{-1}\), penalizing shorter-scale variation.  The radius \(A\) in
\(\cF_\beta(A,\ell)\) is measured in \(\norm{\cdot}_{\beta,\ell}\).  Because
\(\norm{f}_2\leq\norm{f}_{\beta,\ell}\), every
\(f\in\cF_\beta(A,\ell)\) also satisfies \(\norm{f}_2\leq A\), independently
of \(\ell\).  A scale-dependent convention replaces \(A^2\) in
\eqref{eq:target-class} by \(A^2\ell^d\), or equivalently replaces \(A\)
throughout by \(A\ell^{d/2}\).  We use the global-radius convention so that
\(\ell\) controls the variation scale but not the maximal \(L^2\) amplitude.
Finally, the assumption \(\beta>d/2\) ensures that point evaluation is
continuous, so functions in \(\cF_\beta(A,\ell)\) can be sampled at the design
points.

The fill distance, separation radius, and mesh ratio of \(X_N\) are
\begin{equation}\label{eq:geometry}
 h:=\sup_{x\in\T^d}\min_i\dist(x,x_i),
 \qquad
 q:=\frac12\min_{i\neq j}\dist(x_i,x_j),
 \qquad
 \gamma:=h/q.
\end{equation}
We assume that the design is quasi-uniform: its mesh ratio satisfies
\(\gamma\leq\bar\gamma\) for a fixed constant \(\bar\gamma\).  Consequently,
\(N\asymp h^{-d}\).  This assumption allows jittered
lattices, maximin designs, and many deterministic scattered designs; none of
the results requires an exact grid.

The three spatial scales have different roles.  The fill distance \(h\) is
the resolution of the design.  By \eqref{eq:target-class}, \(\ell\) is the
target variation scale; the Mat\'ern multiplier in
\eqref{eq:matern-spectrum} changes at frequency \(\rho^{-1}\), so \(\rho\)
is the kernel lengthscale.  Thus \(h/\ell\) and \(h/\rho\) measure target
and kernel resolution.  When either \(\ell\) or \(\rho\) is at least \(h\),
quasi-uniformity implies that a region with diameter comparable to that
lengthscale contains order \(N\ell^d\) or \(N\rho^d\) design points,
respectively.

Throughout, \(c,C>0\) and constants implicit in \(\lesssim\), \(\gtrsim\),
and \(\asymp\) may change from line to line.  Unless noted, they depend only
on \(d,\beta,\tau,\bar\gamma\); all named constants have the same uniformity
unless a different dependence is stated.  We suppress the dependence of the risk symbols on
\(A,\ell,X_N,\sigma\).

\section{Main results}\label{sec:main-results}

We first establish the minimax benchmark, which contains the target-resolution
and target-information terms.  We then give sharp interpolation bounds,
adding the kernel-resolution term, and finally characterize the scalar-KRR
oracle, in which noise introduces the corresponding kernel-information term.  We
conclude by explaining how these four terms determine the optimal ridge scale
and the asymmetry between longer and shorter kernel lengthscales.

\subsection{The minimax resolution benchmark}

Recall the target information level
\begin{equation}\label{eq:target-information}
 \Iell:=\frac{A^2N\ell^d}{\sigma^2},
\end{equation}
with \(\Iell=\infty\) when \(\sigma=0\).  Let \(\E_f\) denote expectation
under the data model with target \(f\), and define the unrestricted minimax risk
\begin{equation}\label{eq:minimax-risk}
 \Risk^\star
 :=\inf_{\widehat f}
 \sup_{f\in\cF_\beta(A,\ell)}
 \E_f\bignorm{\widehat f-f}_2^2.
\end{equation}
In words, \(\Risk^\star\) is the smallest worst-case expected squared
\(L^2\) error that any estimator~\(\widehat f\) based on the observations can
achieve over the target class \(\cF_\beta(A,\ell)\).

\begin{theorem}[Scale-explicit minimax risk]\label{thm:minimax}
Fix \(d\in\mathbb N\), \(\beta>d/2\), and \(\bar\gamma\geq1\).  There is
\(h_0>0\) such that, if \(h\leq h_0\), \(\gamma\leq\bar\gamma\), and
\(0<\ell\leq1\), then
\begin{equation}\label{eq:minimax-law}
 \Risk^\star
 \asymp
 A^2\left[1\wedge\max\left\{
 (h/\ell)^{2\beta},
 \Iell^{-2\beta/(2\beta+d)}
 \right\}\right].
\end{equation}
\end{theorem}

The upper bound uses the zero estimator when the bracketed term is of order one
and equal-weight least squares over a suitably chosen low-frequency
trigonometric space otherwise.
The theorem separates geometric and statistical resolution.  Since
\(N\asymp h^{-d}\), its noiseless rate is
\(A^2\min\{1,(N\ell^d)^{-2\beta/d}\}\), which for fixed \(\ell\) gives the
familiar rate \(N^{-2\beta/d}\).  With fixed \(\ell,A,\sigma>0\), the noisy
risk remains of order \(A^2\) until both resolution thresholds are crossed and
then eventually has the familiar rate \(N^{-2\beta/(2\beta+d)}\).

\subsection{Sharp Mat\'ern interpolation}

For \(f\in\cF_\beta(A,\ell)\), let \(\mathcal I_{\rho,X_N}f\) denote the solution of
\eqref{eq:krr} with \(\lambda=0\) and data \(y_i=f(x_i)\); equivalently,
\(\mathcal I_{\rho,X_N}f\) is the minimum-RKHS-norm function interpolating \(f\) on
\(X_N\).

\begin{theorem}[Target and kernel resolution in interpolation]
\label{thm:interpolation}
Fix \(d\in\mathbb N\), \(\tau>d/2\),
\(d/2<\beta\leq2\tau\), and \(\bar\gamma\geq1\).  There is
\(h_0>0\) such that, if
\(h\leq h_0\), \(\gamma\leq\bar\gamma\),
\(h\leq\min\{\ell,\rho\}\), and \(0<\ell,\rho\leq1\), then
\begin{equation}\label{eq:interpolation-law}
 \sup_{f\in\cF_\beta(A,\ell)}
 \bignorm{\mathcal I_{\rho,X_N}f-f}_2^2
 \asymp
 A^2\max\left\{(h/\ell)^{2\beta},(h/\rho)^{4\tau}\right\}.
\end{equation}
\end{theorem}

The theorem identifies two geometric requirements: the design must resolve
both \(\ell\) and \(\rho\).  If \(\rho\geq\ell\), interpolation attains the
minimax geometric rate in \cref{thm:minimax}; if \(\rho<\ell\), the kernel
term can dominate even when the design adequately resolves the target
lengthscale.  The kernel-resolution term is already attained by a constant
target, but the next proposition shows that it is not merely a
missing-intercept effect.

\begin{proposition}[Kernel-resolution lower bound at the target scale]
\label{prop:target-scale-mode}
Under the assumptions of \cref{thm:interpolation}, there exist a frequency
\(\nu_\ell\in\Z^d\) satisfying
\(1\leq\ell|\nu_\ell|\leq2\) and a real
unit-norm sine or cosine mode \(\psi_{\nu_\ell}\) such that
\[
 f_\ell:=A\bigl[1+(2\pi\ell|\nu_\ell|)^2\bigr]^{-\beta/2}
 \psi_{\nu_\ell}\in\cF_\beta(A,\ell)
\]
and
\begin{equation}\label{eq:target-scale-mode}
 \bignorm{\mathcal I_{\rho,X_N}f_\ell-f_\ell}_2^2
 \gtrsim A^2(h/\rho)^{4\tau}.
\end{equation}
\end{proposition}

Here \(|\nu_\ell|\asymp\ell^{-1}\), so the target varies at lengthscale
\(\ell\).  The lower bound holds for every design under the stated mesh-ratio
condition, not merely for a lattice, so it is not an aliasing artifact.

\subsection{The four-term KRR oracle}

Recall the kernel information level
\begin{equation}\label{eq:kernel-information}
 \Irho:=\frac{A^2N\rho^d}{\sigma^2},
\end{equation}
with \(\Irho=\infty\) when \(\sigma=0\).  Just as \(\Iell\) in
\eqref{eq:target-information} measures information at the target lengthscale,
\(\Irho\) measures information at the kernel lengthscale.  Recall that
\(\widehat f_{\lambda,\rho}\) is the KRR estimator defined in \eqref{eq:krr}.
The oracle optimizes \(\lambda\) for a fixed kernel lengthscale \(\rho\).
Allowing also the zero estimator, viewed as the \(\lambda=\infty\) endpoint and having
worst-case squared \(L^2\) risk \(A^2\), we define
the scalar-KRR oracle risk at kernel lengthscale \(\rho\) by
\begin{equation}\label{eq:oracle-risk}
 \Risk_\rho^{\mathrm{or}}
 :=A^2\wedge\inf_{\lambda\geq0}
 \sup_{f\in\cF_\beta(A,\ell)}
 \E_f\bignorm{\widehat f_{\lambda,\rho}-f}_2^2.
\end{equation}

The following theorem characterizes this oracle risk uniformly over the target
and kernel lengthscales and gives the corresponding risk equivalence for every
deterministic ridge parameter whose effective resolution remains below both
lengthscales.

\begin{theorem}[Four-term Mat\'ern KRR oracle]\label{thm:oracle}
Fix \(d\in\mathbb N\), \(\bar\gamma\geq1\), \(\tau>d/2\), and
\(d/2<\beta\leq2\tau\).  There is \(h_0>0\) such that, if
\(h\leq h_0\), \(\gamma\leq\bar\gamma\), and
\(0<\ell,\rho\leq1\), then
\begin{equation}\label{eq:oracle-master}
 \Risk_\rho^{\mathrm{or}}
 \asymp
 A^2\left[1\wedge\max\left\{
 (h/\ell)^{2\beta},
 (h/\rho)^{4\tau},
 \Iell^{-2\beta/(2\beta+d)},
 \Irho^{-4\tau/(4\tau+d)}
 \right\}\right].
\end{equation}
Moreover, define \(\eta=\lambda/\rho^d\) and
\(r=\max\{h,\rho\eta^{1/(2\tau)}\}\) for \(\lambda\geq0\).  There exists
\(c_0>0\) such that, for every deterministic \(\lambda\geq0\) satisfying
\(r\leq\min\{\ell,c_0\rho\}\),
\begin{equation}\label{eq:fixed-ridge-law}
 \sup_{f\in\cF_\beta(A,\ell)}
 \E_f\bignorm{\widehat f_{\lambda,\rho}-f}_2^2
 \asymp
 A^2(r/\ell)^{2\beta}
 +A^2(r/\rho)^{4\tau}
 +\frac{\sigma^2}{Nr^d}.
\end{equation}
\end{theorem}

The fixed-ridge equivalence concerns the range in which the effective
resolution \(r\) remains below both relevant lengthscales.  The cap in
\eqref{eq:oracle-risk} covers the complementary oracle regimes.

For discussion and proofs, we write geometric terms as \(\mathsf G\) and
statistical terms as \(\mathsf S\).
The following bounded forms apply whether or not the corresponding design and
information thresholds have been crossed:
\begin{equation}\label{eq:four-terms}
 \begin{aligned}
  \Gell&:=[1+(\ell/h)^2]^{-\beta},
  &\qquad \Grho&:=[1+(\rho/h)^2]^{-2\tau},\\
  \Sell&:=(1+\Iell)^{-2\beta/(2\beta+d)},
  &\qquad \Srho&:=(1+\Irho)^{-4\tau/(4\tau+d)},
 \end{aligned}
\end{equation}
Denote their maximum by
\begin{equation}\label{eq:oracle-envelope}
 \Phior:=\max\{\Gell,\Grho,\Sell,\Srho\}.
\end{equation}
Thus \cref{thm:oracle} is equivalently
\(\Risk_\rho^{\mathrm{or}}\asymp A^2\Phior\).  As shown in
\cref{thm:minimax}, the target terms \(\Gell\) and \(\Sell\) are unavoidable
for every estimator; \(\Grho\) and
\(\Srho\) are the additional terms specific to scalar-ridge Mat\'ern
regression with kernel lengthscale \(\rho\).  The two pairs are defined
independently:
\(\ell\) does not enter the kernel terms, and \(\rho\) does not enter the
target terms.  The lengthscales interact only through which pair dominates
and, when \(\rho\geq\ell\), through the condition \(\beta\leq2\tau\).

\medskip
\begin{center}
\small
\centering
\begin{tabular}{@{}lll@{}}
\toprule
Term & Meaning & Small-term form \\
\midrule
\(\Gell\) & target scale unresolved by the design
  & \((h/\ell)^{2\beta}\) \\
\(\Grho\) & kernel scale unresolved by the design
  & \((h/\rho)^{4\tau}\) \\
\(\Sell\) & insufficient target-scale information
  & \(\Iell^{-2\beta/(2\beta+d)}\) \\
\(\Srho\) & insufficient kernel-scale information
  & \(\Irho^{-4\tau/(4\tau+d)}\) \\
\bottomrule
\end{tabular}
\end{center}
\medskip

\Cref{fig:oracle-phase-diagram} shows the dominant term on the numerical grid
of \cref{sec:numerics}.
\begin{figure}[!h]
\centering
\includegraphics[width=\textwidth,trim=0 12bp 0 12bp,clip]{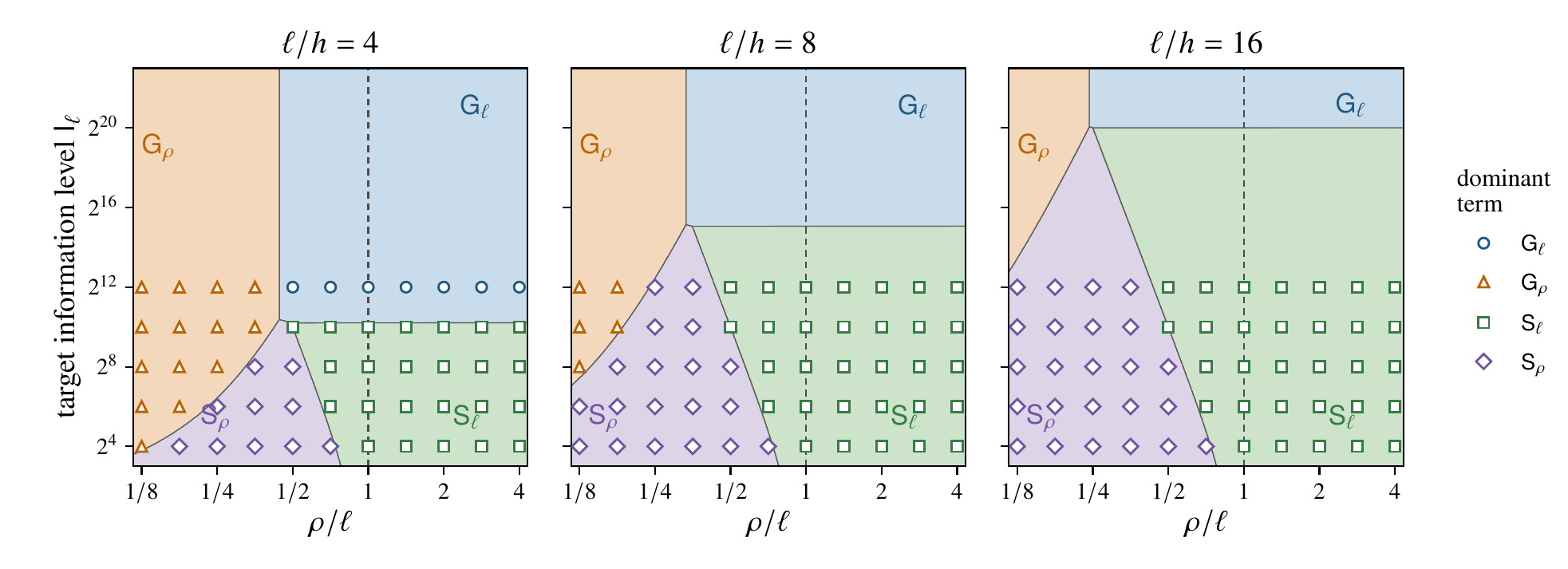}
\caption{Four-term phase diagram for \(d=1\) and \(\beta=\tau=2\).  Colors
identify the largest bounded term in \eqref{eq:four-terms}; markers give the
same theoretical classification on the numerical grid, not an empirical risk
decomposition.  The dashed line marks \(\rho=\ell\).}
\label{fig:oracle-phase-diagram}
\end{figure}

\paragraph{Origin of the four terms}
When all four terms are small, the upper bound is constructive.  Recall that
the normalized ridge parameter and effective resolution are
\begin{equation}\label{eq:eta-r}
 \eta:=\lambda/\rho^d,
 \qquad
 r(\eta):=\max\{h,\rho\eta^{1/(2\tau)}\},
\end{equation}
so increasing \(r\) corresponds to stronger regularization.  The
fixed-ridge equivalence \eqref{eq:fixed-ridge-law} identifies its three
components: target bias, kernel bias, and variance.  Because both bias terms
increase with \(r\), the constraint \(r\geq h\) prevents them from falling
below the geometric terms \(\Gell\) and \(\Grho\).  Balancing variance with
the target and kernel biases gives, respectively,
\begin{equation}\label{eq:balance-scales}
 r_\ell:=\ell\Iell^{-1/(2\beta+d)},
 \qquad
 r_\rho:=\rho\Irho^{-1/(4\tau+d)},
 \qquad
 r_\star:=\max\{h,\min(r_\ell,r_\rho)\}.
\end{equation}
At \(r=r_\ell\), target bias and variance have the common order
\(A^2\Sell\); at \(r=r_\rho\), kernel bias and variance have the common order
\(A^2\Srho\).  As \(r\) increases, variance decreases while both biases
increase, so the smaller of \(r_\ell\) and \(r_\rho\) is the first balance
reached and gives the unconstrained optimal resolution.  The design cannot
support a resolution finer than \(h\), which replaces this choice by
\(r_\star=\max\{h,\min(r_\ell,r_\rho)\}\) and produces the geometric terms
\(\Gell\) and \(\Grho\).  When \(\rho\geq\ell\), the source--resolvent
cancellation in \eqref{eq:resolvent-cancellation} makes kernel bias no larger
than target bias.

\paragraph{Constructive oracle tuning}
The oracle risk in \cref{thm:oracle} is attained within constant factors by
the following deterministic choice of \(\lambda\).  When \(\sigma=0\), set
\(r_\ell=r_\rho=0\), so \(r_\star=h\).  When all four terms in
\eqref{eq:four-terms} are sufficiently small, choose
\begin{equation}\label{eq:oracle-tuning}
 \eta_\star=
 \begin{cases}
 0,&r_\star=h,\\
 (r_\star/\rho)^{2\tau},&r_\star>h,
 \end{cases}
 \qquad
 \lambda_\star=\rho^d\eta_\star.
\end{equation}
Otherwise, use the zero estimator.  The precise threshold conditions are
verified in \cref{sec:proof-oracle}.
When \(\sigma=0\), this construction gives \(r_\star=h\) and
\(\eta_\star=0\), so when \(h\leq\min\{\ell,\rho\}\) the oracle and
interpolation bounds agree.

\begin{corollary}[Minimax comparison]\label{cor:minimax-comparison}
Under the assumptions of \cref{thm:minimax,thm:oracle},
\begin{equation}\label{eq:minimax-comparison}
 \frac{\Risk_\rho^{\mathrm{or}}}{\Risk^\star}
 \asymp
 1\vee
 \frac{\Grho\vee\Srho}{\Gell\vee\Sell}.
\end{equation}
Consequently, scalar-ridge Mat\'ern KRR has minimax order if and only if
\(\Grho\vee\Srho\lesssim\Gell\vee\Sell\).  In the noiseless resolved
regime of \cref{thm:interpolation}, the Mat\'ern interpolant has minimax order
if and only if \(\Grho\lesssim\Gell\).
In particular, \(\rho\geq\ell\) guarantees both conclusions.
\end{corollary}

\begin{proof}
Using the bounded terms in \eqref{eq:four-terms}, \cref{thm:minimax} can be
written as
\(\Risk^\star\asymp A^2(\Gell\vee\Sell)\).  Dividing this identity into
the corresponding form of \cref{thm:oracle} gives
\eqref{eq:minimax-comparison}.  If
\(\rho\geq\ell\), then \(\beta\leq2\tau\) gives
\begin{equation}\label{eq:overscaled-domination}
 \Grho\leq\Gell,
 \qquad
 \Srho\leq\Sell.
\end{equation}
The interpolation statement follows in the same way from
\cref{thm:minimax,thm:interpolation} with \(\sigma=0\).
\end{proof}

\section{Proofs of the main results}\label{sec:proofs}

The main ingredients for each proof are summarized below.
\medskip
\begin{center}
\small
\setlength{\tabcolsep}{5pt}
\begin{tabular}{@{}p{0.14\textwidth}p{0.39\textwidth}p{0.39\textwidth}@{}}
\toprule
Result & Upper-bound argument & Lower-bound arguments \\
\midrule
Minimax & spectral least squares & null function; Assouad \\
Interpolation & source--resolvent cancellation & null function; low-mode leakage \\
KRR oracle & resolved Mat\'ern bias; variance upper bound
  & minimax; target- and kernel-scale bias lower bounds; variance lower bound \\
\bottomrule
\end{tabular}
\end{center}
\medskip
\Cref{sec:proof-prelim} collects the preliminary estimates;
\cref{sec:proof-minimax,sec:proof-interpolation,sec:proof-oracle} then prove
the minimax, interpolation, and KRR oracle results, respectively.

Throughout this section, the design conditions \(h\leq h_0\) and
\(\gamma\leq\bar\gamma\), and the scale bounds \(0<\ell,\rho\leq1\), are
imposed whenever the corresponding quantities occur.  All Sobolev indices
used in sampling estimates exceed \(d/2\).

\subsection{Preliminary estimates}\label{sec:proof-prelim}

We first fix the normalized kernel and operator notation used in the estimates
below.  Write
\(
 \bar\mu_{\rho,k}:=\rho^{-d}\mu_{\rho,k}
 =(1+(2\pi\rho|k|)^2)^{-\tau}
\).  Let \(\bar K_{\rho,\tau}\) be the kernel with Fourier coefficients
\(\bar\mu_{\rho,k}\), and let \(\bar T_{\rho,\tau}\) be its integral operator.
We use the notation
\[
 \norm u_{s,\ell_0}^2:=
 \sum_{k\in\Z^d}(1+(2\pi\ell_0|k|)^2)^s|u_k^{\ft}|^2,
 \qquad s\geq0,\quad \ell_0>0.
\]
This extends the target norm in \eqref{eq:target-class}; the normalized
Mat\'ern RKHS norm is \(\norm{\cdot}_{\tau,\rho}\).  Let
\(\bar\cH_{\rho,\tau}\) be the RKHS of \(\bar K_{\rho,\tau}\), let
\(J:\bar\cH_{\rho,\tau}\to L^2(\T^d)\) be inclusion, and equip
\(E_N:=\mathbb C^N\) with
\[
 \ip{z}{\zeta}_N
 :=N^{-1}\sum_{i=1}^Nz_i\overline{\zeta_i},
 \qquad \norm z_N^2:=\ip{z}{z}_N.
\]
Define \(S:\bar\cH_{\rho,\tau}\to E_N\) by \((Su)_i=u(x_i)\), and put
\(C_X:=S^*S\).  Let \(\bar{\boldsymbol K}_{\rho,X}\) be the sampled matrix
of \(\bar K_{\rho,\tau}\).  The normalized data-space Gram operator and
recovery map are
\[
 G_X:=SS^*=\bar{\boldsymbol K}_{\rho,X}/N,
 \qquad
 Q_\eta:=J(C_X+\eta I)^{-1}S^*,
 \qquad \eta=\lambda/\rho^d,
\]
with the Moore--Penrose interpretation at \(\eta=0\).  Thus
\(\widehat f_{\lambda,\rho}=Q_\eta(f|_{X_N}+\xi)\).  In particular, for
noiseless data, \(Q_0(f|_{X_N})=\mathcal I_{\rho,X_N}f\).  The covariance operator
of \(\xi\) in \(E_N\) is \((\sigma^2/N)I\); the factor \(N^{-1}\) comes
from the normalized inner product on \(E_N\).  Here \(C_X\) acts on the
RKHS and \(G_X\) on data space, and they have the same nonzero eigenvalues.
For \(K\geq1\), let
\[
 V_K:=\operatorname{span}\{e_k:\norm k_\infty\leq K\}.
\]
Set
\begin{equation}\label{eq:delta-main}
 \delta_{h,\rho}:=[1+(\rho/h)^2]^{-\tau}.
\end{equation}

\begin{lemma}[Quasi-uniform sampling estimates]
\label{lem:sampling-main}
Under the standing assumptions, there is \(c_{\mathrm{MZ}}>0\)
such that the low-frequency
Marcinkiewicz--Zygmund (MZ) inequality holds:
\begin{equation}\label{eq:mz}
 \norm{p|_{X_N}}_N^2\asymp\norm p_2^2,
 \qquad p\in V_K,\quad Kh\leq c_{\mathrm{MZ}}.
\end{equation}
For each \(s\in\{\beta,\tau\}\), general Sobolev functions satisfy the
trace estimate
\begin{equation}\label{eq:trace}
 \norm{g|_{X_N}}_N^2
 \lesssim \norm g_2^2+h^{2s}\norm g_{H^s}^2,
 \qquad g\in H^s(\T^d).
\end{equation}
\end{lemma}

\begin{lemma}[Reverse sampling and Gram stability]
\label{lem:reverse-gram-main}
Under the standing assumptions, every \(u\in H^\tau(\T^d)\) satisfies
\begin{equation}\label{eq:reverse-sampling-main}
 \norm u_2^2\lesssim
 \delta_{h,\rho}\norm u_{\tau,\rho}^2
 +\norm{u|_{X_N}}_N^2,
\end{equation}
and the normalized Gram matrix \(G_X\) is invertible and satisfies
\begin{equation}\label{eq:gram-main}
 \lambda_{\min}(G_X)\gtrsim\delta_{h,\rho}.
\end{equation}
\end{lemma}

\begin{corollary}[Recovery stability and source resolvent]
\label{cor:recovery-stability-main}
Under the hypotheses of \cref{lem:reverse-gram-main}, uniformly in
\(\eta\geq0\),
\begin{equation}\label{eq:data-stability-main}
 \norm{Q_\eta z}_2\lesssim\norm z_N,
 \qquad z\in E_N.
\end{equation}
If \(v\in L^2(\T^d)\) and \(p=\bar T_{\rho,\tau}v\), then
\begin{equation}\label{eq:source-resolvent-main}
 \bignorm{p-Q_\eta(p|_{X_N})}_2
 \lesssim \max\{\delta_{h,\rho},\eta\}\norm v_2.
\end{equation}
If \(h\leq\rho\) and
\(r=\max\{h,\rho\eta^{1/(2\tau)}\}\), then
\[
 \bignorm{p-Q_\eta(p|_{X_N})}_2
 \lesssim (r/\rho)^{2\tau}\norm v_2.
\]
\end{corollary}

\Cref{lem:sampling-main} supplies the sampling input for spectral least
squares, while \cref{lem:reverse-gram-main,cor:recovery-stability-main}
supply the stability and resolvent estimates for interpolation and KRR.
Their complete proofs are given, in the same order, in
\cref{sec:sm-auxiliary}.

\subsection{Proof of the minimax theorem}\label{sec:proof-minimax}

The proof has one constructive upper-bound argument and two lower-bound
arguments.  Spectral
least squares gives the upper bound, a target-null-function construction gives the
target-resolution term, and an Assouad construction gives the target-scale
information term due to noise.  The final bandwidth choice matches the larger
of the two lower bounds.

Let \(\widehat f_K\) be equal-weight least squares over the real-valued
functions in \(V_K\).

\begin{lemma}[Spectral least-squares upper bound]\label{lem:ls-main}
If \(K\geq1\) and \(Kh\leq c_{\mathrm{MZ}}\), then
\begin{equation}\label{eq:ls-oracle}
 \sup_{f\in\cF_\beta(A,\ell)}
 \E_f\bignorm{\widehat f_K-f}_2^2
 \lesssim
 A^2\{1+(\ell K)^2\}^{-\beta}
 +A^2(h/\ell)^{2\beta}
 +\frac{\sigma^2}{N}K^d
 .
\end{equation}
\end{lemma}

\begin{proof}
Let \(\widetilde p_K\) be the least-squares fit to the noiseless samples
\(f|_{X_N}\).  Let \(P_K\) be the Fourier-coordinate projection onto \(V_K\),
and write \(f=p_K+g_K\), where \(p_K=P_Kf\).  By linearity,
\[
 \begin{aligned}
 \widehat f_K-f
 &= (\widetilde p_K-p_K)-g_K+(\widehat f_K-\widetilde p_K),\\
 \E_f\bignorm{\widehat f_K-f}_2^2
 &=\norm{g_K}_2^2+\norm{\widetilde p_K-p_K}_2^2
   +\E\bignorm{\widehat f_K-\widetilde p_K}_2^2.
 \end{aligned}
\]
The second line uses \(g_K\perp V_K\) and the centering of the fitted-noise
term.  We bound the three terms on its right-hand side in turn.  Fourier
truncation gives
\begin{equation}\label{eq:ls-tail-main}
 \norm{g_K}_2^2\lesssim A^2\{1+(\ell K)^2\}^{-\beta},
 \qquad
 \norm{g_K}_{H^\beta}^2\lesssim A^2\ell^{-2\beta}.
\end{equation}
The trace estimate \eqref{eq:trace} therefore implies
\begin{equation}\label{eq:ls-sample-tail-main}
 \norm{g_K|_{X_N}}_N^2
 \lesssim A^2\left[\{1+(\ell K)^2\}^{-\beta}
 +(h/\ell)^{2\beta}\right].
\end{equation}
For noiseless data, \((\widetilde p_K-p_K)|_{X_N}\) is the empirical
orthogonal projection of \(g_K|_{X_N}\) onto \(V_K|_{X_N}\), and hence
\(\bignorm{(\widetilde p_K-p_K)|_{X_N}}_N
\leq\norm{g_K|_{X_N}}_N\).  The lower MZ inequality in \eqref{eq:mz}
then bounds \(\bignorm{\widetilde p_K-p_K}_2^2\) by
\eqref{eq:ls-sample-tail-main}.

For the noise contribution, choose an \(L^2\)-orthonormal real basis
\(\{\varphi_j\}_{j=1}^{M_K}\) of \(V_K\), where \(M_K\asymp K^d\), and
put \(B_{ij}=\varphi_j(x_i)\) and \(H_K=N^{-1}B^\top B\).  By \eqref{eq:mz},
\(H_K\asymp I\) in the Loewner order.  The coefficient vector of
\(\widehat f_K-\widetilde p_K\) is \(N^{-1}H_K^{-1}B^\top\xi\), and hence
\[
 \E\bignorm{N^{-1}H_K^{-1}B^\top\xi}_{\ell^2}^2
 =\frac{\sigma^2}{N}\tr(H_K^{-1})
 \lesssim\frac{\sigma^2K^d}{N}.
\]
Combining the three estimates proves
\eqref{eq:ls-oracle}.
\end{proof}

\begin{lemma}[Target-resolution null function]
\label{lem:target-hole-main}
There exists a real \(g\in\cF_\beta(A,\ell)\) such that
\(g|_{X_N}=0\) and
\begin{equation}\label{eq:target-null-main}
 \norm g_2^2\gtrsim A^2\Gell.
\end{equation}
Consequently, every estimator satisfies
\begin{equation}\label{eq:target-hole-main}
 \sup_{f\in\cF_\beta(A,\ell)}
 \E_f\bignorm{\widehat f-f}_2^2
 \gtrsim A^2\Gell.
\end{equation}
\end{lemma}

\begin{proof}
Choose \(x_\star\in\T^d\) with \(\dist(x_\star,X_N)=h\) and a nonzero real function
\(\phi\in C_c^\infty(B(0,1/3))\).  For sufficiently small \(h\), periodize
\[
 g_0(x)=a_h\sum_{n\in\Z^d}\phi((x-x_\star+n)/h).
\]
Its support misses \(X_N\).  Fourier scaling gives
\[
 \norm{g_0}_2^2=c_\phi a_h^2h^d,
 \qquad
 \norm{g_0}_{\beta,\ell}^2
 \leq C_\phi a_h^2h^d[1+(\ell/h)^2]^\beta.
\]
Choose
\(a_h\asymp A h^{-d/2}[1+(\ell/h)^2]^{-\beta/2}\), with a sufficiently
small fixed comparison factor, and denote the resulting function by \(g\).
Then \(g\in\cF_\beta(A,\ell)\),
\(g|_{X_N}=0\), and
\(\norm g_2^2\gtrsim A^2[1+(\ell/h)^2]^{-\beta}=A^2\Gell\).
The data laws under \(g\) and \(-g\) coincide, and the parallelogram identity
gives, for every estimator,
\[
 \frac12\E_g\bignorm{\widehat f-g}_2^2
 +\frac12\E_{-g}\bignorm{\widehat f+g}_2^2
 =\E_0\norm{\widehat f}_2^2+\norm g_2^2.
\]
Thus its worst-case risk is at least \(\norm g_2^2\), proving the claim.
\end{proof}

\begin{lemma}[Statistical information lower bound]
\label{lem:statistical-lower-main}
Every estimator satisfies
\begin{equation}\label{eq:statistical-lower-main}
 \sup_{f\in\cF_\beta(A,\ell)}
 \E_f\bignorm{\widehat f-f}_2^2
 \gtrsim A^2\Sell.
\end{equation}
\end{lemma}

\begin{proof}
The claim is immediate when \(\sigma=0\), since \(\Sell=0\).  Suppose
\(\sigma>0\) and set
\[
 L_\ell:=1\vee\Iell^{1/(2\beta+d)},
 \qquad K:=\lceil L_\ell/\ell\rceil.
\]
Choose a block \(\mathcal K_K\subset\Z^d\) of
\(M\asymp K^d\) frequencies satisfying
\(K\leq|k|\leq C_dK\) and containing no pair \(\{k,-k\}\).  Then
\(\varphi_k(x)=\sqrt2\cos(2\pi k\cdot x)\), \(k\in\mathcal K_K\), are
real, orthonormal, and uniformly bounded.
Since \(L_\ell\geq1\) and \(\ell\leq1\),
\(L_\ell\leq\ell K\leq L_\ell+\ell\leq2L_\ell\).  Thus
\(\ell K\asymp L_\ell\) and \(M\asymp\ell^{-d}L_\ell^d\).
Put \(W_K=[1+C_d(\ell K)^2]^\beta\).  For
\(\omega\in\{-1,1\}^{\mathcal K_K}\), let
\[
 f_\omega=\varepsilon_K
 \sum_{k\in\mathcal K_K}\omega_k\varphi_k,
 \qquad
 \varepsilon_K=c_1A(MW_K)^{-1/2}.
\]
Since
\(\norm{f_\omega}_{\beta,\ell}^2
 \leq \varepsilon_K^2MW_K=c_1^2A^2\),
this hypercube lies in \(\cF_\beta(A,\ell)\) when \(c_1\leq1\).
Let \(P_\omega\) denote the corresponding data law.  If two vertices differ
in one coordinate, their Kullback--Leibler divergence satisfies
\[
 \operatorname{KL}(P_\omega,P_{\omega'})
 \lesssim \frac{N\varepsilon_K^2}{\sigma^2}
 \lesssim c_1^2\Iell(\ell K)^{-d}W_K^{-1}.
\]
By the definition of \(L_\ell\), the last quantity is
\(\lesssim c_1^2\Iell L_\ell^{-(2\beta+d)}\lesssim c_1^2\), and is therefore uniformly
small when \(c_1\) is small.  Moreover,
\[
 M\varepsilon_K^2
 =c_1^2A^2W_K^{-1}
 \asymp A^2(1+L_\ell^2)^{-\beta}
 \asymp A^2(1\wedge\Iell^{-2\beta/(2\beta+d)})
 \asymp A^2\Sell.
\]
For any estimator, estimate each unknown sign \(\omega_k\) by the sign of
\(\operatorname{Re}\bigl\langle\widehat f,\varphi_k\bigr\rangle_{L^2}\),
and denote it by \(\widehat\omega_k\).  If
\(\widehat\omega_k\neq\omega_k\), the estimated coefficient does not have
the true sign, so its distance from \(\varepsilon_K\omega_k\) is at least
\(\varepsilon_K\).  Summing these squared coefficient errors over the
orthonormal block gives
\(\bignorm{\widehat f-f_\omega}_2^2
\geq\varepsilon_K^2\sum_{k\in\mathcal K_K}
\boldsymbol1_{\{\widehat\omega_k\ne\omega_k\}}\).
Assouad's lemma, together with Pinsker's inequality in the
KL-to-total-variation step
\cite[Lem.~2.5 and Lem.~2.12]{tsybakov2009nonparametric}, therefore gives,
for every estimator,
\[
 \sup_{f\in\cF_\beta(A,\ell)}
 \E_f\bignorm{\widehat f-f}_2^2
 \gtrsim M\varepsilon_K^2\asymp A^2\Sell.
\]
No condition \(K\lesssim h^{-1}\) is needed: aliasing on the prescribed
design can only make the experiments harder to distinguish.
\end{proof}

\begin{proof}[Proof of \cref{thm:minimax}]
The lower bound follows by combining
\cref{lem:target-hole-main,lem:statistical-lower-main}: every estimator has
risk \(\gtrsim A^2(\Gell\vee\Sell)\).

For the upper bound, fix a sufficiently small constant
\(\varepsilon_{\mathrm{mm}}>0\).  If
\(\max\{\Gell,\Sell\}\geq\varepsilon_{\mathrm{mm}}\), the zero estimator
gives the required upper bound.  Otherwise,
both \(\ell/h\) and \(\Iell\) are bounded below by fixed large constants.
Put \(K_\star=\ell^{-1}\Iell^{1/(2\beta+d)}\), with
\(K_\star=\infty\) when \(\sigma=0\), and choose
\begin{equation}\label{eq:minimax-bandwidth-main}
 K\asymp
 \min\{h^{-1},K_\star\},
\end{equation}
with a sufficiently small comparison factor that
\(Kh\leq c_{\mathrm{MZ}}\) in
\cref{lem:ls-main}.  If \(K_\star\lesssim h^{-1}\), take
\(K\asymp K_\star\).  The Fourier bias and variance in
\eqref{eq:ls-oracle} are then both of order
\(A^2\Iell^{-2\beta/(2\beta+d)}\), while
\(K_\star\lesssim h^{-1}\) implies
\((h/\ell)^{2\beta}\lesssim\Iell^{-2\beta/(2\beta+d)}\).
Thus the upper bound is \(\lesssim A^2\Sell\).

If \(K_\star\gtrsim h^{-1}\), take \(K\asymp h^{-1}\).  The first two
terms in \eqref{eq:ls-oracle} are \(\lesssim A^2\Gell\), and
\(K_\star\gtrsim h^{-1}\) gives
\[
 \frac{\sigma^2K^d}{N}
 \asymp \frac{A^2}{\Iell}(\ell/h)^d
 \lesssim A^2(h/\ell)^{2\beta}.
\]
Hence this branch is \(\lesssim A^2\Gell\).  In the present nontrivial
regime, \(\Gell\asymp(h/\ell)^{2\beta}\) and
\(\Sell\asymp\Iell^{-2\beta/(2\beta+d)}\).  Combining the two upper
bounds with the lower bound above proves the theorem in all regimes.
\end{proof}

\subsection{Proof of the interpolation theorem}\label{sec:proof-interpolation}

For the upper bound, decompose \(f=(f-p_r)+p_r\), where
\(p_r=P_rf\) retains the frequencies below the cutoff \(r^{-1}\).  Target
smoothness controls the discarded tail \(f-p_r\).  Because \(p_r\) is band
limited and the Mat\'ern multiplier never vanishes, it can be written as
\(p_r=\bar T_{\rho,\tau}v_r\).  The source norm \(\|v_r\|_2\) measures the
representation cost of undoing the Mat\'ern multiplier, while the
source--resolvent estimate supplies the compensating recovery factor.  The lower bound
combines the target-resolution construction from \cref{lem:target-hole-main} with a
Fourier-frame argument showing that Mat\'ern translates with a small kernel
lengthscale necessarily leak energy away from a prescribed low mode.

\begin{lemma}[Band-limited source representation]\label{lem:source-main}
Let \(h\leq r\leq\min\{\ell,\rho\}\), where here \(r\) is an arbitrary
cutoff scale.  Let \(P_r\) be the Fourier projection onto
\(\{|k|\leq r^{-1}\}\), and set \(p_r=P_rf\).  For every
\(f\in\cF_\beta(A,\ell)\), there is \(v_r\in L^2(\T^d)\) such that
\(p_r=\bar T_{\rho,\tau}v_r\) and
\begin{equation}\label{eq:source-size}
 \bignorm{f-p_r}_2+\bignorm{(f-p_r)|_{X_N}}_N
 \lesssim A(r/\ell)^\beta.
\end{equation}
If \(\ell\leq\rho\) and \(\beta\leq2\tau\), then
\begin{equation}\label{eq:long-source-size}
 \norm{v_r}_2
 \lesssim A(\rho/r)^{2\tau}(r/\ell)^\beta,
\end{equation}
whereas, if \(\rho\leq\ell\), then
\begin{equation}\label{eq:short-source-size}
 \norm{v_r}_2
 \lesssim A\left\{1\vee
 (\rho/r)^{2\tau}(r/\ell)^\beta\right\}.
\end{equation}
\end{lemma}

\begin{proof}
The target constraint gives
\(\norm{f-p_r}_2\lesssim A(r/\ell)^\beta\) and
\(\norm{f-p_r}_{H^\beta}\lesssim A\ell^{-\beta}\).  Since \(h\leq r\),
the trace estimate \eqref{eq:trace} gives the same bound for the sampled tail,
proving \eqref{eq:source-size}.  Define \(v_r\) by
\[
 (v_r)_k^{\ft}
 =\begin{cases}
 (1+(2\pi\rho|k|)^2)^\tau f_k^{\ft},&|k|\leq r^{-1},\\
 0,&|k|>r^{-1}.
 \end{cases}
\]
Thus \(v_r\) is obtained by undoing the Mat\'ern multiplier on the retained
frequencies, and \(\bar T_{\rho,\tau}v_r=p_r\).
Define the multiplier ratio
\[
 \mathcal M(t):=\frac{(1+\rho^2t)^\tau}{(1+\ell^2t)^{\beta/2}},
 \qquad t\geq0.
\]
By Parseval and the defining constraint of \(\cF_\beta(A,\ell)\),
\begin{align*}
 \norm{v_r}_2^2
 &=\sum_{|k|\leq r^{-1}}\mathcal M((2\pi|k|)^2)^2
   [1+(2\pi\ell|k|)^2]^\beta |f_k^{\ft}|^2\\
 &\leq A^2\sup_{0\leq t\leq4\pi^2r^{-2}}\mathcal M(t)^2.
\end{align*}
Thus it remains to maximize \(\mathcal M\) over the cutoff interval.  A direct
calculation gives
\[
 (\log\mathcal M)'(t)
 =\frac{\tau\rho^2-\frac{\beta}{2}\ell^2
 +(\tau-\frac{\beta}{2})\rho^2\ell^2t}
 {(1+\rho^2t)(1+\ell^2t)}.
\]
Writing \(t_r=4\pi^2r^{-2}\), the condition
\(r\leq\min\{\ell,\rho\}\) gives
\(\mathcal M(0)=1\) and
\(\mathcal M(t_r)\asymp(\rho/r)^{2\tau}(r/\ell)^\beta\).
If \(\ell\leq\rho\) and \(\beta\leq2\tau\), the numerator is
nonnegative, so \(\mathcal M\) is nondecreasing and evaluation at the cutoff proves
\eqref{eq:long-source-size}.  If \(\rho\leq\ell\) and
\(\beta\leq2\tau\), the numerator is nondecreasing, so any interior
critical point is a minimum.  If \(\rho\leq\ell\) and \(\beta>2\tau\),
the numerator is negative and \(\mathcal M\) is decreasing.  In either short-kernel
case, the maximum is attained either at \(t=0\) or at the cutoff, which proves
\eqref{eq:short-source-size}.
\end{proof}

\begin{lemma}[Resolved Mat\'ern bias]\label{lem:bias-main}
For \(\eta\geq0\), let
\(r=\max\{h,\rho\eta^{1/(2\tau)}\}\leq\min\{\ell,\rho\}\).  If
\(\ell\leq\rho\) and \(d/2<\beta\leq2\tau\), then
\begin{equation}\label{eq:long-bias-main}
 \sup_{f\in\cF_\beta(A,\ell)}
 \bignorm{f-Q_\eta(f|_{X_N})}_2
 \lesssim A(r/\ell)^\beta.
\end{equation}
If \(\rho\leq\ell\), then
\begin{equation}\label{eq:short-bias-main}
 \sup_{f\in\cF_\beta(A,\ell)}
 \bignorm{f-Q_\eta(f|_{X_N})}_2
 \lesssim A\left\{(r/\ell)^\beta+(r/\rho)^{2\tau}\right\}.
\end{equation}
\end{lemma}

\begin{proof}
Take \(p_r\) and \(v_r\) from \cref{lem:source-main}.  By
\eqref{eq:data-stability-main}, the first two terms in
\[
 f-Q_\eta(f|_{X_N})
 =(f-p_r)-Q_\eta((f-p_r)|_{X_N})
   +p_r-Q_\eta(p_r|_{X_N})
\]
are bounded by \eqref{eq:source-size}.  Since
\(r=\max\{h,\rho\eta^{1/(2\tau)}\}\) and \(h\leq\rho\),
\(\max\{\delta_{h,\rho},\eta\}\lesssim(r/\rho)^{2\tau}\).
If \(\ell\leq\rho\), \eqref{eq:source-resolvent-main} and
\eqref{eq:long-source-size} give the cancellation
\begin{equation}\label{eq:resolvent-cancellation}
 \bignorm{p_r-Q_\eta(p_r|_{X_N})}_2
 \lesssim (r/\rho)^{2\tau}\norm{v_r}_2
 \lesssim A(r/\ell)^\beta,
\end{equation}
which proves \eqref{eq:long-bias-main}.  If \(\rho\leq\ell\),
\eqref{eq:source-resolvent-main} and \eqref{eq:short-source-size} give
\[
 \bignorm{p_r-Q_\eta(p_r|_{X_N})}_2
 \lesssim A\left\{(r/\ell)^\beta+(r/\rho)^{2\tau}\right\},
\]
and the same three-term decomposition proves \eqref{eq:short-bias-main}.
\end{proof}

The scale dependence in \eqref{eq:resolvent-cancellation} is the central
cancellation.  Representing the cutoff target through
\(\bar T_{\rho,\tau}\) costs \((\rho/r)^{2\tau}\) in source norm, while the
recovery resolvent contributes \((r/\rho)^{2\tau}\).  If
\(\ell\leq\rho\), these factors cancel and leave the target error
\(A(r/\ell)^\beta\).  If \(\rho<\ell\), the first term \(1\) in the maximum
in \eqref{eq:short-source-size} survives the same resolvent estimate
and leaves
\(A(r/\rho)^{2\tau}\); its square is the kernel-bias term.  At \(\eta=0\),
recall that \(Q_0(f|_{X_N})=\mathcal I_{\rho,X_N}f\), so \cref{lem:bias-main}
supplies the interpolation upper bound.  For \(\eta>0\), the same lemma
supplies the deterministic bias used in the resolved fixed-ridge risk equivalence,
\cref{lem:resolved-krr-main}.

\begin{lemma}[Finite-exclusion Fourier frame]\label{lem:frame}
Let \(F\subset\Z^d\) be finite.  There are constants \(h_F,L_F,c>0\), depending
only on \(d,\bar\gamma\), and \(|F|\), such that, if
\(h\leq h_F\), \(\gamma\leq\bar\gamma\), then, for every
\(\alpha\in\mathbb C^N\),
\begin{equation}\label{eq:frame}
 \sum_{\substack{|k|\leq L_F/q\\k\notin F}}
 \left|\sum_{j=1}^N\alpha_j e^{-2\pi i k\cdot x_j}\right|^2
 \geq cq^{-d}\norm{\alpha}_{\ell^2}^2.
\end{equation}
\end{lemma}

A complete proof of \cref{lem:frame} is given in \cref{sec:sm-frame}.

\begin{lemma}[Low-mode leakage]\label{lem:low-mode-leakage-main}
Uniformly for every \(\nu\in\Z^d\),
\begin{equation}\label{eq:mode-leakage}
 \operatorname{dist}_{L^2}\!\left(e_\nu,
 \operatorname{span}\{\bar K_{\rho,\tau}(\cdot-x_j)\}_{j=1}^N\right)^2
 \gtrsim [1+(\rho/h)^2]^{-2\tau}.
\end{equation}
\end{lemma}

\begin{proof}
Since the constants in \cref{lem:frame} depend on \(F\) only through
\(|F|\), decrease the standing \(h_0\) once so that its threshold holds
uniformly for singletons, and apply it with \(F=\{\nu\}\).  For
\(g_\alpha=\sum_j\alpha_j\bar K_{\rho,\tau}(\cdot-x_j)\), define
\(\mathcal A_k=\sum_j\alpha_j e^{-2\pi i k\cdot x_j}\), so
\((g_\alpha)_k^{\ft}=\bar\mu_{\rho,k}\mathcal A_k\).  On the frequency
window in \eqref{eq:frame},
\(\bar\mu_{\rho,k}^2\gtrsim[1+(\rho/q)^2]^{-2\tau}
\gtrsim[1+(\rho/h)^2]^{-2\tau}\), since \(q\geq h/\bar\gamma\).
Combining this bound with \eqref{eq:frame}, and then using
\(|\mathcal A_\nu|^2\leq N\norm{\alpha}_{\ell^2}^2\) and
\(N\asymp q^{-d}\), gives
\[
 \sum_{\substack{|k|\leq L_F/q\\k\neq\nu}}
 \bar\mu_{\rho,k}^2|\mathcal A_k|^2
 \gtrsim [1+(\rho/h)^2]^{-2\tau}q^{-d}
 \norm{\alpha}_{\ell^2}^2
 \gtrsim [1+(\rho/h)^2]^{-2\tau}|\mathcal A_\nu|^2.
\]
Parseval and this estimate imply
\begin{equation}
\begin{aligned}
 \bignorm{e_\nu-g_\alpha}_2^2
 &= |1-\bar\mu_{\rho,\nu}\mathcal A_\nu|^2
 +\sum_{k\neq\nu}\bar\mu_{\rho,k}^2|\mathcal A_k|^2 \\
 &\geq |1-\bar\mu_{\rho,\nu}\mathcal A_\nu|^2
 +\sum_{\substack{|k|\leq L_F/q\\k\neq\nu}}
 \bar\mu_{\rho,k}^2|\mathcal A_k|^2 \\
 &\geq |1-\bar\mu_{\rho,\nu}\mathcal A_\nu|^2
 +c[1+(\rho/h)^2]^{-2\tau}|\mathcal A_\nu|^2.
\end{aligned}
\label{eq:leakage-quadratic}
\end{equation}
Write \(b:=c[1+(\rho/h)^2]^{-2\tau}\), decreasing \(c\) so that
\(b\leq1\).  The final line of
\eqref{eq:leakage-quadratic} is a quadratic in
\(z=\mathcal A_\nu\in\mathbb C\).  Its exact minimum is
\(b/(\bar\mu_{\rho,\nu}^2+b)\), which is at least \(b/2\) because
\(0<\bar\mu_{\rho,\nu}\leq1\).  Taking the infimum over \(\alpha\) proves
\eqref{eq:mode-leakage}.
\end{proof}

The quadratic form \eqref{eq:leakage-quadratic} displays the mechanism.  A
kernel combination either misses the prescribed low Fourier coefficient, or
uses coefficients with nontrivial \(\mathcal A_\nu\).  In the second case,
\cref{lem:frame} forces comparable coefficient energy into other modes with
frequencies of order at most \(h^{-1}\).  The Mat\'ern multipliers on that
band turn this forced energy into squared \(L^2\) error of order
\([1+(\rho/h)^2]^{-2\tau}\), which is \((h/\rho)^{4\tau}\) when
\(h\leq\rho\).  The frame argument uses only separation and bounded mesh
ratio, so the leakage is not a lattice aliasing effect.

\begin{proof}[Proof of \cref{prop:target-scale-mode}]
Let \(\mathbf e^{(1)}:=(1,0,\ldots,0)\in\Z^d\) and take
\(\nu_\ell=\lceil\ell^{-1}\rceil\mathbf e^{(1)}\).  Then
\(1\leq\ell|\nu_\ell|\leq2\).  Denote its normalized real cosine and sine
modes by \(\psi_{\nu_\ell}^{\cos}\) and \(\psi_{\nu_\ell}^{\sin}\).  Since
the interpolation error operator is real and
\(e_{\nu_\ell}=(\psi_{\nu_\ell}^{\cos}
+i\psi_{\nu_\ell}^{\sin})/\sqrt2\),
\[
 2\bignorm{\mathcal I_{\rho,X_N}e_{\nu_\ell}-e_{\nu_\ell}}_2^2
 =\bignorm{\mathcal I_{\rho,X_N}\psi_{\nu_\ell}^{\cos}-\psi_{\nu_\ell}^{\cos}}_2^2
 +\bignorm{\mathcal I_{\rho,X_N}\psi_{\nu_\ell}^{\sin}-\psi_{\nu_\ell}^{\sin}}_2^2.
\]
Hence \eqref{eq:mode-leakage} forces at least one of those two real modes,
denoted by \(\psi_{\nu_\ell}\), to have squared interpolation error
\(\gtrsim[1+(\rho/h)^2]^{-2\tau}\).  The displayed scaling in
\cref{prop:target-scale-mode} puts that mode in
\(\cF_\beta(A,\ell)\), while
\([1+(2\pi\ell|\nu_\ell|)^2]^{-\beta}\asymp1\).  Since \(h\leq\rho\), this proves
\eqref{eq:target-scale-mode}.
\end{proof}

\begin{proof}[Proof of \cref{thm:interpolation}]
For the upper bound set \(\eta=0\), hence \(r=h\), in
\cref{lem:bias-main}.  If \(\ell\leq\rho\), the kernel term is dominated
because \(\beta\leq2\tau\); if \(\rho\leq\ell\),
\eqref{eq:short-bias-main} gives both terms.  The target lower bound is
\cref{lem:target-hole-main} with \(\sigma=0\).  The preceding result shows
that the kernel obstruction is not confined to constants; for the full kernel
lower bound, take \(\nu=0\) in \cref{lem:low-mode-leakage-main} and scale the
constant target by \(A\).  This proves \eqref{eq:interpolation-law}.
\end{proof}

\subsection{Proof of the KRR oracle theorem}\label{sec:proof-oracle}

\Cref{lem:fixed-target-bias-main,lem:ridge-shrinkage-main,lem:variance-lower-main}
give the target-scale bias, kernel-scale bias, and variance lower bounds;
\cref{lem:bias-main,lem:variance-upper-main} give the matching upper bounds.
The resolved fixed-ridge risk equivalence assembles them, after which optimizing the
effective resolution proves the oracle theorem.

Throughout this subsection,
\(r=\max\{h,\rho\eta^{1/(2\tau)}\}\) as in \eqref{eq:eta-r}; centered
noise gives \(\E_f\widehat f_{\lambda,\rho}=Q_\eta(f|_{X_N})\), the
operator form used in the bias estimates below.
For \(z,\zeta\in E_N\), write \(z\otimes_N\zeta\) for the rank-one map
\(w\mapsto\ip{w}{\zeta}_Nz\).
For \(z_k=e_k|_{X_N}\), the data-space Gram operator has the Fourier
decomposition
\begin{equation}\label{eq:gram-fourier-main}
 G_X=\sum_{k\in\Z^d}\bar\mu_{\rho,k}z_k\otimes_Nz_k.
\end{equation}

\begin{lemma}[Target-scale bias lower bound]\label{lem:fixed-target-bias-main}
For every \(\eta\geq0\), if \(r\leq\ell\), then
\begin{equation}\label{eq:fixed-target-bias-main}
 \sup_{f\in\cF_\beta(A,\ell)}
 \bignorm{f-Q_\eta(f|_{X_N})}_2^2
 \gtrsim A^2(r/\ell)^{2\beta}.
\end{equation}
\end{lemma}

\begin{proof}
If \(r=h\), use the target-hole function from
\cref{lem:target-hole-main}.  Here \(h=r\leq\ell\).  Its samples vanish,
so its KRR output is zero,
and its squared \(L^2\) norm is \(\gtrsim A^2(h/\ell)^{2\beta}\).

Suppose \(r>h\).  Then \(\eta>0\) and
\(r=\rho\eta^{1/(2\tau)}\).  Let
\(\nu=\lceil r^{-1}\rceil\mathbf e^{(1)}\).  Since \(r\leq\ell\leq1\),
\(|\nu|\asymp r^{-1}\) and \(\norm{z_\nu}_N=1\).  If \(r\leq\rho\), then
\(\rho|\nu|\asymp\rho/r\), so the Mat\'ern multiplier gives the following
bound; if \(r>\rho\), it follows
instead from \(\bar\mu_{\rho,\nu}\leq1\leq(r/\rho)^{2\tau}\).  Thus
\begin{equation}\label{eq:moving-mode-weight-main}
 \bar\mu_{\rho,\nu}\lesssim(r/\rho)^{2\tau}=\eta.
\end{equation}
Split \eqref{eq:gram-fourier-main} as
\[
 G_X=\bar\mu_{\rho,\nu}z_\nu\otimes_Nz_\nu+B_\nu,
 \qquad B_\nu\succeq0,
\]
and put
\(\zeta_\nu=\ip{z_\nu}{(B_\nu+\eta I)^{-1}z_\nu}_N\leq\eta^{-1}\).
Here \(B_\nu\) collects the contributions from all Fourier modes other than
\(\nu\).  Applying the rank-one resolvent formula to the preceding
decomposition gives
\[
\begin{aligned}
 (Q_\eta z_\nu)_\nu^{\ft}
 &=\bar\mu_{\rho,\nu}
   \ip{z_\nu}{(G_X+\eta I)^{-1}z_\nu}_N
   =\frac{\bar\mu_{\rho,\nu}\zeta_\nu}
          {1+\bar\mu_{\rho,\nu}\zeta_\nu},\\
 1-(Q_\eta z_\nu)_\nu^{\ft}
 &\geq\frac{\eta}{\eta+\bar\mu_{\rho,\nu}}\gtrsim1.
\end{aligned}
\]
By the same real sine--cosine identity used in the proof of
\cref{prop:target-scale-mode}, at least one normalized real mode
\(\psi_\nu\) at frequency \(\nu\) has squared bias \(\gtrsim1\).  Scaling it by
\(A[1+(2\pi\ell|\nu|)^2]^{-\beta/2}\) places it in
\(\cF_\beta(A,\ell)\); because \(|\nu|\asymp r^{-1}\), its squared
amplitude is \(\gtrsim A^2(r/\ell)^{2\beta}\).  This proves
\eqref{eq:fixed-target-bias-main}.
\end{proof}

\begin{lemma}[Kernel-scale bias lower bound]\label{lem:ridge-shrinkage-main}
For every \(\eta\geq0\),
\begin{equation}\label{eq:ridge-all-scale-main}
 \sup_{f\in\cF_\beta(A,\ell)}
 \bignorm{f-Q_\eta(f|_{X_N})}_2^2
 \gtrsim A^2\left\{
 [1+(\rho/h)^2]^{-2\tau}
 \vee\left(\frac{\eta}{1+\eta}\right)^2\right\}.
\end{equation}
Moreover, there is \(c_0>0\) such that, whenever
\(r\leq c_0\rho\),
\begin{equation}\label{eq:ridge-low-mode}
 \sup_{f\in\cF_\beta(A,\ell)}
 \bignorm{f-Q_\eta(f|_{X_N})}_2^2
 \gtrsim A^2(r/\rho)^{4\tau}.
\end{equation}
\end{lemma}

\begin{proof}
Let \(\mathbf1:=(1,\ldots,1)\in E_N\).  Decompose the normalized sampled
Gram operator as
\begin{equation}\label{eq:rank-one-split}
 G_X=\mathbf1\otimes_N\mathbf1+B,
 \qquad B\succeq0.
\end{equation}
For the constant target \(f\equiv1\) and \(\eta>0\), Sherman--Morrison shows that
the recovered constant Fourier coefficient is
\begin{equation}\label{eq:sherman}
 \theta_\eta=\frac{\zeta_\eta}{1+\zeta_\eta},
 \qquad
 \zeta_\eta=\ip{\mathbf1}{(B+\eta I)^{-1}\mathbf1}_N
 \leq\eta^{-1}.
\end{equation}
Therefore \(1-\theta_\eta\geq\eta/(1+\eta)\).  Scaling the data and target
by \(A\), the zero Fourier coefficient gives the ridge term in
\eqref{eq:ridge-all-scale-main}.  Since \(Q_\eta\mathbf1\) belongs to
\(\operatorname{span}\{\bar K_{\rho,\tau}(\cdot-x_j):1\leq j\leq N\}\),
\cref{lem:low-mode-leakage-main} with \(\nu=0\) gives the first, leakage
term in \eqref{eq:ridge-all-scale-main}, including \(\eta=0\).
If \(r\leq c_0\rho\), then
\([1+(\rho/h)^2]^{-2\tau}\asymp(h/\rho)^{4\tau}\) and
\((\eta/(1+\eta))^2\asymp\eta^2\); their maximum is
\((r/\rho)^{4\tau}\).  This proves \eqref{eq:ridge-low-mode}.
\end{proof}

\begin{lemma}[Variance upper bound]
\label{lem:variance-upper-main}
For every \(\eta\geq0\), if \(r\leq1\), then
\begin{equation}\label{eq:variance-upper-main}
 \E\norm{Q_\eta\xi}_2^2
 \lesssim\frac{\sigma^2}{Nr^d}.
\end{equation}
\end{lemma}

\begin{proof}
A Fourier-cutoff effective-dimension argument is given in
\cref{sec:sm-variance-upper}.
\end{proof}

\begin{lemma}[Variance lower bound]
\label{lem:variance-lower-main}
There is \(c_0>0\) such that, for every \(\eta\geq0\), if
\(h\leq\rho\leq1\) and \(r\leq c_0\rho\), then
\begin{equation}\label{eq:variance-lower}
 \norm{Q_\eta}_{\mathrm{HS}}^2\gtrsim r^{-d},
 \qquad
 \E\norm{Q_\eta\xi}_2^2
 \gtrsim\frac{\sigma^2}{Nr^d}.
\end{equation}
\end{lemma}

\begin{proof}
Let \(V_K=\operatorname{span}\{e_k:\norm k_\infty\leq K\}\), whose
frequencies satisfy \(|k|\leq\sqrt d\,K\), and take
\(K=\lfloor\kappa/r\rfloor\), where \(0<\kappa<c_{\mathrm{MZ}}\) is fixed
and sufficiently small.  We show that \(Q_\eta S\) transmits this
low-frequency space with uniformly positive gain.  Fix \(p\in V_K\) and
write \(p=\bar T_{\rho,\tau}v\).  Fourier diagonalization and the definition
of \(r\) give, respectively,
\[
 \norm v_2\leq[1+C_d(\rho K)^2]^\tau\norm p_2,
 \qquad
 \max\{\delta_{h,\rho},\eta\}\lesssim(r/\rho)^{2\tau}.
\]
Here the first bound uses \(|k|\leq\sqrt d\,K\), and the second uses
\(h\leq\rho\).  Combining these bounds with the source-resolvent estimate
\eqref{eq:source-resolvent-main} and \(K\leq\kappa/r\) yields
\[
 \bignorm{p-Q_\eta(p|_{X_N})}_2
 \lesssim (r/\rho)^{2\tau}
 [1+C_d(\rho\kappa/r)^2]^\tau\norm p_2
 \lesssim [(r/\rho)^2+C_d\kappa^2]^\tau\norm p_2.
\]
Choose first \(\kappa\) and then \(c_0\leq\min\{\kappa/2,1\}\) so that the
right-hand side is at most \(\frac12\norm p_2\) whenever \(r\leq c_0\rho\).
Then \(K\asymp r^{-1}\) and \(\dim(V_K)\asymp r^{-d}\).
Endow \(V_K\) with its \(L^2\) inner product and define
\(U_K:=Q_\eta S|_{V_K}:V_K\to L^2(\T^d)\).  The preceding error bound
implies
\begin{align*}
 \norm{U_Kp}_2&\geq\tfrac12\norm p_2
       &&\text{for every }p\in V_K,\\
 U_K^*U_K&\succeq\tfrac14I_{V_K},
 &\qquad
 \norm{U_K}_{\mathrm{HS}}^2
       &=\tr(U_K^*U_K)
         \geq\tfrac14\dim(V_K)\asymp r^{-d}.
\end{align*}
Since \(r\geq h\), we have \(Kh\leq\kappa<c_{\mathrm{MZ}}\), and the upper
bound in \eqref{eq:mz} gives \(\norm{S|_{V_K}}_{\mathrm{op}}\lesssim1\).
The Hilbert--Schmidt ideal inequality \cite[Thm.~2.7(a)]{simon2005trace}
therefore gives
\[
 r^{-d/2}\lesssim\norm{U_K}_{\mathrm{HS}}
 \leq\norm{Q_\eta}_{\mathrm{HS}}
      \norm{S|_{V_K}}_{\mathrm{op}}
 \lesssim\norm{Q_\eta}_{\mathrm{HS}}.
\]
This proves \(\norm{Q_\eta}_{\mathrm{HS}}^2\gtrsim r^{-d}\).  Finally,
\(\E\norm{Q_\eta\xi}_2^2=(\sigma^2/N)\norm{Q_\eta}_{\mathrm{HS}}^2
\gtrsim\sigma^2/(Nr^d)\), proving both assertions in
\eqref{eq:variance-lower}.
\end{proof}

Thus KRR transmits an \(r^{-d}\)-dimensional subspace with uniformly positive
gain, matching the variance upper bound in \cref{lem:variance-upper-main}.

\begin{lemma}[Resolved fixed-ridge risk equivalence]\label{lem:resolved-krr-main}
There is \(c_0>0\) such that, for every \(\eta\geq0\),
with \(\lambda=\rho^d\eta\), if
\(r\leq\min\{\ell,c_0\rho\}\), then
\begin{equation}\label{eq:resolved-fixed-ridge-main}
 \sup_{f\in\cF_\beta(A,\ell)}
 \E_f\bignorm{\widehat f_{\lambda,\rho}-f}_2^2
 \asymp
 A^2(r/\ell)^{2\beta}
 +A^2(r/\rho)^{4\tau}
 +\frac{\sigma^2}{Nr^d}.
\end{equation}
\end{lemma}

\begin{proof}
Decrease \(c_0\), if necessary, so that the local conclusions of
\cref{lem:ridge-shrinkage-main,lem:variance-lower-main} hold with the same
constant.  For the upper bound, add the variance estimate
\eqref{eq:variance-upper-main} to the appropriate bias estimate in
\cref{lem:bias-main}.  If \(\ell\leq\rho\), then
\((r/\rho)^{4\tau}\leq(r/\ell)^{2\beta}\), so the displayed kernel term is
already dominated by the target bound for \(\ell\leq\rho\).

For the lower bound, linearity and centered noise give
\[
 \sup_{f\in\cF_\beta(A,\ell)}
 \E_f\bignorm{\widehat f_{\lambda,\rho}-f}_2^2
 =\sup_{f\in\cF_\beta(A,\ell)}
 \bignorm{f-Q_\eta(f|_{X_N})}_2^2
 +\E\norm{Q_\eta\xi}_2^2.
\]
The target and kernel bias lower bounds are
\eqref{eq:fixed-target-bias-main} and \eqref{eq:ridge-low-mode}.  Their targets
may differ, but the supremum is at least half the sum of the two bounds; adding
\eqref{eq:variance-lower} proves \eqref{eq:resolved-fixed-ridge-main}.
\end{proof}

\begin{proof}[Proof of \cref{thm:oracle}]
The fixed-ridge assertion \eqref{eq:fixed-ridge-law} is
\cref{lem:resolved-krr-main}.  It remains to optimize over ridge and handle
the remaining cases in which a geometric or information term is of order one.

\emph{Lower bound.}
The minimax theorem gives
a lower bound \(\gtrsim A^2(\Gell\vee\Sell)\) for every estimator.  If
\(\rho\geq\ell\), these
terms also dominate \(\Grho\vee\Srho\) by
\eqref{eq:overscaled-domination}.  Suppose therefore that \(\rho\leq\ell\).
The kernel-scale bias bound \eqref{eq:ridge-all-scale-main} gives a lower
bound \(\gtrsim A^2\Grho\) for every ridge parameter.  If both \(h/\rho\) and
\(\eta^{1/(2\tau)}\) are below a sufficiently small constant, then
\eqref{eq:resolved-fixed-ridge-main} gives
\begin{equation}\label{eq:kernel-objective-lower-main}
 \sup_{f\in\cF_\beta(A,\ell)}
 \E_f\bignorm{\widehat f_{\lambda,\rho}-f}_2^2
 \gtrsim A^2(r/\rho)^{4\tau}
 +\frac{\sigma^2}{Nr^d}.
\end{equation}
Using \(\sigma^2/(Nr^d)=A^2\Irho^{-1}(\rho/r)^d\), the two terms balance at
\(r_\rho=\rho\Irho^{-1/(4\tau+d)}\), where both have order
\(A^2\Irho^{-4\tau/(4\tau+d)}\).  If \(r_\rho<h\), the constraint \(r\geq h\)
instead yields \(A^2(h/\rho)^{4\tau}\).  If \(r_\rho>c_0\rho\), with
\(c_0\) as in \cref{lem:resolved-krr-main}, then
\(\Irho<c_0^{-(4\tau+d)}\), so \(\Srho\) is bounded below by a fixed positive
constant.  For \(r\leq c_0\rho\), the variance term in
\eqref{eq:kernel-objective-lower-main} satisfies
\[
 \frac{\sigma^2}{Nr^d}
 =A^2\Irho^{-1}(\rho/r)^d\gtrsim A^2.
\]
If \(r>c_0\rho\), then either \(h>c_0\rho\) or
\(\eta^{1/(2\tau)}>c_0\); the leakage or ridge-shrinkage term in
\eqref{eq:ridge-all-scale-main} is then bounded below by a fixed positive
constant.  Thus every ridge parameter has risk \(\gtrsim A^2\Srho\).  In
all cases the kernel terms therefore contribute
\(\gtrsim A^2(\Grho\vee\Srho)\).
This proves all four oracle lower terms.

\emph{Upper bound.}
Let \(c_0\leq1\) be the constant in \cref{lem:resolved-krr-main}, decreased
if necessary, and choose \(\varepsilon_0>0\) sufficiently small in terms of
\(c_0\).  If \(\Phior>\varepsilon_0\), the zero estimator has risk at most
\(A^2\lesssim A^2\Phior\).  It therefore remains only to consider
\(\Phior\leq\varepsilon_0\).  The definitions of the four terms then imply
\[
 h\leq c_0\min\{\ell,\rho\},
 \qquad r_\ell\leq c_0\ell,
 \qquad r_\rho\leq c_0\rho,
 \qquad \Iell,\Irho\geq1.
\]
The definition of \(r_\star\) and the preceding bounds give
\[
 r_\star=\max\{h,\min(r_\ell,r_\rho)\}
 \leq c_0\min\{\ell,\rho\}
 \leq\min\{\ell,c_0\rho\},
\]
so \eqref{eq:resolved-fixed-ridge-main} applies to \(\lambda_\star\), for
which \(r=r_\star\).

If \(\rho\leq\ell\), choose \(r=r_\star\) and
\(\lambda=\lambda_\star\) in \eqref{eq:resolved-fixed-ridge-main}.  Put
\(r_0=\min\{r_\ell,r_\rho\}\).  At \(r_0\), the bias associated with the
smaller balance scale equals the variance, while the other bias is no larger.
If \(h\leq r_0\), all three terms are therefore bounded by
\(A^2(\Sell\vee\Srho)\).  If \(h>r_0\), replacing \(r_0\) by
\(r_\star=h\) decreases the variance, while the two biases are bounded by
\(A^2(\Gell\vee\Grho)\).  Hence the right-hand side is
\(\lesssim A^2\Phior\), proving the branch \(\rho\leq\ell\).

If \(\rho\geq\ell\), the kernel terms satisfy
\(\Grho\leq\Gell\) and \(\Srho\leq\Sell\) by
\eqref{eq:overscaled-domination}.  If \(\sigma=0\), then \(r_\star=h\) by
definition.  If \(\sigma>0\), the two balance scales obey
\[
 \frac{r_\rho}{r_\ell}
 =\left(\frac{\rho}{\ell}\right)^{4\tau/(4\tau+d)}
 \Iell^{(4\tau-2\beta)/((2\beta+d)(4\tau+d))}
 \geq1,
\]
where the last inequality uses \(\Iell\geq1\) and \(\beta\leq2\tau\).
Hence in both cases the unified choice is
\(r_\star=h\vee r_\ell\).  Applying
\eqref{eq:resolved-fixed-ridge-main} and using
\((r_\star/\rho)^{4\tau}\leq(r_\star/\ell)^{2\beta}\), which follows from
\(r_\star\leq\ell\leq\rho\) and \(\beta\leq2\tau\), gives
\[
 \sup_{f\in\cF_\beta(A,\ell)}
 \E_f\bignorm{\widehat f_{\lambda_\star,\rho}-f}_2^2
 \lesssim A^2(r_\star/\ell)^{2\beta}
 +\frac{\sigma^2}{Nr_\star^d}
 \lesssim A^2(\Gell\vee\Sell).
\]
Thus the same tuning proves the branch \(\rho\geq\ell\) without a
kernel-dependent statistical penalty and completes the proof.
\end{proof}

\section{Numerical experiments}\label{sec:numerics}

We use deterministic finite-section calculations to test the predicted phase
structure without selecting favorable targets or simulating noise.  Here a
finite section means restricting the target and kernel Fourier expansions to
finitely many modes.  It is a spectral truncation of the infinite-dimensional
function space, distinct from the finite sample size \(N\) and from a spatial
grid discretization.  The principal experiments take \(d=1\), \(N=64\),
\(A=1\), and \(\beta=\tau=2\).  We compare
the regular lattice with the bounded jitter
\[
 x_i=(i+1/2+\epsilon_i)/N \pmod 1,
 \qquad \epsilon_i\sim\operatorname{Unif}[-0.2,0.2],
\]
using one fixed realization throughout.  Their measured mesh ratios are \(1\)
and \(2.14\), respectively, computed with the periodic torus distance.  All
dimensionless parameters below use the measured fill distance of the
corresponding design.

The experiments examine the interpolation and KRR consequences separately.
\Cref{fig:interpolation-numerics} isolates the kernel-resolution term in
interpolation, while \cref{fig:krr-numerics} compares the grid-refined oracle
with the four-term envelope, evaluates the tuning obtained from the
fixed-ridge balance, and contrasts the oracle-risk plateau for
\(\rho\geq\ell\) with numerical conditioning.  A focused \(d=2\) experiment
testing the dimension dependence of the two information exponents is reported
in \cref{sec:sm-dimension-numerics}.

\paragraph{Finite-section calculation}
On each finite Fourier section, we construct the KRR map exactly.  Its squared
bias is a quadratic form in the target coefficients, so the worst bias over
the truncated target ball is the largest generalized eigenvalue relative to
the target Sobolev weight.  The variance is the corresponding exact trace, so
no target is selected and no noise is simulated.  Enlarging all cutoffs by a
factor \(1.5\) changes selected risks by at most \(0.4\%\) and the largest
reported condition numbers by at most \(0.25\%\).  Taking \(K_T\geq2/h\)
changes a \(\Gell\)-dominant oracle risk by at most \(0.08\%\) and the plateau
in \cref{fig:interpolation-numerics}(b) by at most \(0.98\%\) on both designs.
For \(d=1\), let \(G_{X,K}\) denote the normalized Gram matrix at cutoff
\(K\), and let \(\kappa_2\) denote its spectral condition number.
The cutoffs and numerical-reliability criterion are
\begin{equation}\label{eq:numerical-cutoffs}
 \begin{gathered}
 K_T=\max\{8,\lceil5/\ell\rceil\},\qquad
 K=\max\{K_T,2N,\lceil8/\rho\rceil\},\\
 \kappa_2(G_{X,K}+\eta I)\epsilon_{\mathrm{mach}}\leq10^{-8}.
 \end{gathered}
\end{equation}
Target modes use \(|k|\leq K_T\), while Gram and recovery modes use
\(|k|\leq K\).  With \(\epsilon_{\mathrm{mach}}\approx2.2\times10^{-16}\),
we retain only solves satisfying \eqref{eq:numerical-cutoffs}, add no nugget,
and write \(\eta_{\mathrm{or}}\) for the grid-refined minimizer below.

For KRR, we minimize the sum of the computed bias and variance over a fixed
grid in \(\log_{10}\eta\), with spacing refined to \(0.05\) near each minimum,
and include both \(\eta=0\) and the zero estimator.  No selected positive-ridge
minimum lies at an outer reliable-grid boundary.  We also evaluate the choice
in \eqref{eq:balance-scales}--\eqref{eq:oracle-tuning}.

\paragraph{Interpolation transition}
The main interpolation experiment fixes \(\ell/h=64\) and takes
\(\rho/h\in\{1,\sqrt2,2,2\sqrt2,4,4\sqrt2,8\}\).  In
\cref{fig:interpolation-numerics}(a), multiplying \(\rho/h\) by \(\sqrt2\)
reduces the error by a factor approaching \(16\) on both designs, in agreement
with the exponent \(4\tau=8\).  At \(\ell/h=8\), the dashed vertical line in
\cref{fig:interpolation-numerics}(b) marks where the kernel-resolution term
\((h/\rho)^{4\tau}\) equals the target-resolution term
\((h/\ell)^{2\beta}\), namely
\(\rho/\ell=(\ell/h)^{\beta/(2\tau)-1}=8^{-1/2}\).  To the left, the kernel
term is larger and the error decreases as \(\rho\) increases; to the right,
the target term is larger, so further increasing \(\rho\) does not improve
the predicted order and the curves flatten.  Comparison constants may shift
the observed transition.

\begin{figure}[!h]
\centering
\includegraphics[width=\textwidth]{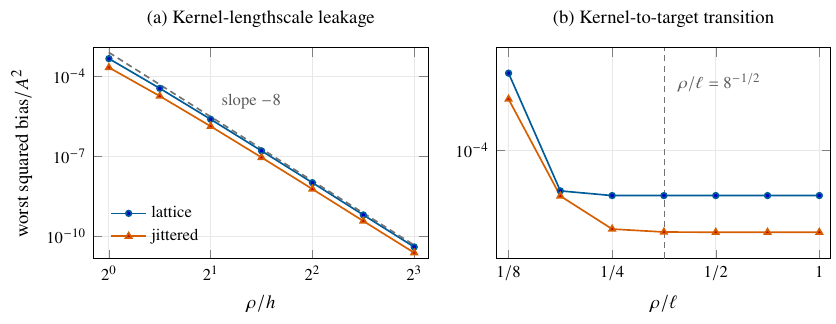}
\caption{Finite-section worst-class Mat\'ern interpolation error for \(d=1\)
and \(\beta=\tau=2\).  (a) Kernel-lengthscale dependence at \(\ell/h=64\),
with dashed slope \(-8\).  (b) Transition at \(\ell/h=8\); the vertical line
equates the two resolution contributions.  Circles and triangles denote the
lattice and jittered designs.}
\label{fig:interpolation-numerics}
\end{figure}

\paragraph{Four-term KRR oracle}
We use
\[
 \ell/h\in\{4,8,16\},\qquad
 \rho/\ell=2^{j/2},\quad -6\leq j\leq4,
 \qquad
 \Iell\in\{16,64,256,1024,4096\},
\]
with \(\sigma^2=N\ell/\Iell\).  \Cref{fig:oracle-phase-diagram} shows the
dominant term for all \(330\) parameter--design combinations, with every term
active somewhere.  In \cref{fig:krr-numerics}(a), the ratio of grid-refined
risk to \(A^2\Phior\) lies
in \(0.219\)--\(0.267\) for \(\Gell\), \(0.016\)--\(0.304\) for \(\Grho\),
\(0.374\)--\(0.491\) for \(\Sell\), and \(0.294\)--\(0.406\) for \(\Srho\).
All \(330\) selected minima are reliable; fixed-ridge tuning is within a factor
\(2.91\) of the grid-refined oracle.

At \(\ell/h=8\) and \(\Iell=256\), increasing \(\rho/\ell\) from
\(2^{-1/2}\) to \(4\) changes the lattice oracle risk by less than \(4.3\%\)
but raises \(\kappa_2(G_{X,K})\) from \(3.15\times10^3\) to
\(3.16\times10^6\) and the oracle-shifted condition number from about \(70\)
to \(5.53\times10^4\); see \cref{fig:krr-numerics}(b)--(c).  This separates
statistical accuracy from translate-basis conditioning.  For
\(\rho/\ell\geq1\), the flat risk agrees with
\eqref{eq:overscaled-domination}; \(\Risk^\star\) is not computed.

\begin{figure}[!htb]
\centering
\includegraphics[width=.96\textwidth]{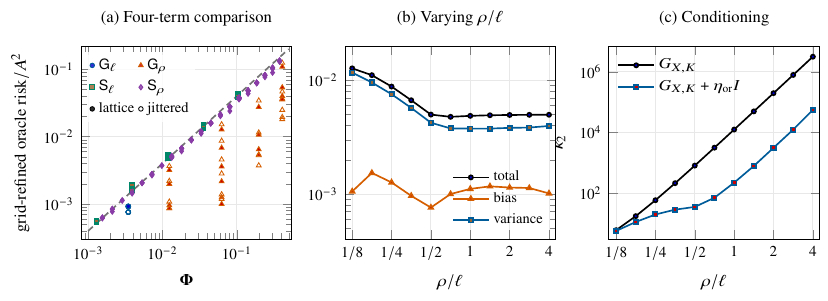}
\caption{Four-term KRR oracle and conditioning for \(d=1\) and
\(\beta=\tau=2\).  (a) Grid-refined oracle risk versus \(\Phior\); shape
identifies the dominant term and fill the design, with a slope-one guide.
(b) Oracle risk, bias, and variance at \(\ell/h=8\), \(\Iell=256\).
(c) Corresponding Gram and oracle-shifted condition numbers.}
\label{fig:krr-numerics}
\end{figure}

\paragraph{Statistical accuracy versus conditioning}
The harmlessness of \(\rho\geq\ell\) is a statistical statement in exact
arithmetic.  Because large \(\rho/h\) can make the kernel matrix severely
ill-conditioned, stable linear algebra or preconditioning may be needed to
realize the theoretical risk.  This flat-kernel regime is studied in
\cite{barthelmeusevich2021flat}; we do not claim a scale-uniform condition
number or analyze a particular preconditioner.

\paragraph{Code and data availability}
The complete repository is available at
\href{https://github.com/sanzalonso/matern-lengthscale-resolution}
{\nolinkurl{github.com/}}\linebreak
\href{https://github.com/sanzalonso/matern-lengthscale-resolution}
{\nolinkurl{sanzalonso/matern-lengthscale-resolution}}.\par\smallskip

\FloatBarrier
\section{Discussion}\label{sec:discussion}

\paragraph{Resolution before smoothness}
\Cref{thm:minimax,thm:interpolation,thm:oracle} identify a hierarchy of
effects: \(h/\ell,h/\rho\) govern geometric resolution, \(\Iell,\Irho\)
govern information, and only after these thresholds do \(\beta\) and \(\tau\)
govern decay.  The target terms \(\Gell,\Sell\) are unavoidable, whereas
\(\Grho,\Srho\) quantify the additional cost of the chosen kernel.  The
comparison is asymmetric: when \(\rho\geq\ell\) and \(\beta\leq2\tau\),
optimally tuned KRR attains minimax order; when \(\rho<\ell\), kernel-scale
geometry or information creates an additional barrier, while longer kernels
may remain poorly conditioned; see \cref{fig:krr-numerics}(c).

\paragraph{Scope}
The torus removes boundary issues but retains scattered bounded-mesh-ratio
designs; bounded domains or non-quasi-uniform sets require boundary-adapted
spaces or local resolution.  Replacing \(q\) by \(h\) costs factors
\(\bar\gamma^{2\tau}\) and \(\bar\gamma^{4\tau}\) in the Gram and leakage
bounds.  The condition \(\beta\leq2\tau\) controls the long-kernel source
multiplier; for \(\beta>2\tau\), ordinary smoothness saturation must be tracked.
Anisotropy, nonstationarity, other kernels, data-driven tuning, and other noise
models are outside scope.

\paragraph{Main takeaway}
The four terms separate target limits from kernel costs: resolution and
information determine whether error can decrease, and smoothness its rate.

\newpage
\section*{Acknowledgments}
The author was partly funded by NSF CAREER Award DMS-2237628.  The author used
OpenAI Codex to assist with proof auditing
and numerical code development.  The author assumes responsibility for all
content.

\bibliographystyle{siamplain}
\bibliography{resolution_matern_references}

\clearpage
\renewcommand{\siampretitle}{Supplementary Materials: }
\renewcommand{\siamprelabel}{SM}
\renewcommand{\thepage}{\siamprelabel\arabic{page}}
\setcounter{page}{1}
\setcounter{section}{0}
\setcounter{subsection}{0}
\setcounter{subsubsection}{0}
\setcounter{paragraph}{0}
\setcounter{subparagraph}{0}
\setcounter{equation}{0}
\setcounter{theorem}{0}
\setcounter{figure}{0}
\setcounter{table}{0}
\def\theHsection{SM.\arabic{section}}
\def\theHsubsection{SM.\arabic{section}.\arabic{subsection}}
\def\theHsubsubsection{SM.\arabic{section}.\arabic{subsection}.\arabic{subsubsection}}
\def\theHequation{SM.\arabic{section}.\arabic{equation}}
\def\theHtheorem{SM.\arabic{section}.\arabic{theorem}}
\def\theHfigure{SM.\arabic{figure}}
\def\theHtable{SM.\arabic{table}}
\def\theHfootnote{SM.\arabic{footnote}}
\headers{Non-asymptotic Mat\'ern Regression}{D. Sanz-Alonso}
\pdfbookmark[0]{Supplementary Materials}{supplementary-materials}
\makeatletter
\let\maketitle\combinedsavedmaketitle
\let\@maketitle\combinedsavedatmaketitle
\let\thanks\combinedsavedthanks
\makeatother
\title{Non-asymptotic Analysis of Mat\'ern Regression:\\
The Roles of Target and Kernel Lengthscales}
\author{Daniel Sanz-Alonso\thanks{\raggedright Department of Statistics,
University of Chicago, Chicago, IL 60637 USA.\\
\email{sanzalonso@uchicago.edu}.}}
\renewcommand{\siampretitle}{Supplementary Materials\\}
\maketitle
\renewcommand{\siampretitle}{Supplementary Materials: }

We use the notation and standing assumptions of
\cref{sec:setting,sec:proof-prelim} in the main paper.  In particular,
\(h,q,\gamma,\bar\gamma,E_N,\norm{\cdot}_N,V_K\), the normalized Mat\'ern
objects, \(J,S,C_X,G_X,Q_\eta,\eta,r\), and \(\delta_{h,\rho}\) retain their
meanings there.  Unless noted, constants implicit in \(\lesssim\),
\(\gtrsim\), and \(\asymp\) have the dependence stated in the main-paper
result being proved or in the supplementary result stated below.  For
functions on \(\mathbb R^d\), we use the Fourier convention
\(u^{\ft}(\xi)=\int_{\mathbb R^d}u(x)e^{-2\pi i\xi\cdot x}\dd x\).

\section{Proofs of auxiliary estimates}
\label{sec:sm-auxiliary}

We prove the three preliminary results from the main paper in order.
Quasi-uniform geometry yields the sampling and trace estimates, which imply
reverse sampling.  Reverse sampling gives the source-resolvent estimate and,
together with the intrinsic Gram bound, recovery stability.  We then prove
the variance upper bound used in the KRR argument.

\begin{proof}[Proof of \Cref{lem:sampling-main}]
First decrease the standing \(h_0\), if necessary, so that \(2h_0\) is below
the injectivity radius; each \(B(x,2h_0)\) then admits a Euclidean coordinate
lift with uniform constants.
Fix \(s\in\{\beta,\tau\}\).  Standard packing and covering estimates for a
bounded-mesh-ratio design give \(h/\bar\gamma\leq q\leq h\),
\(N^{-1}\lesssim q^d\), and bounded overlap of the balls
\(B(x_i,2h)\).  If \(V_i\) is the periodic Voronoi cell of \(x_i\) and
\(w_i:=|V_i|\), the same estimates and the inclusions
\(B(x_i,q)\subset V_i\subset B(x_i,h)\) give
\begin{equation}\label{eq:sm-voronoi-weights}
 w_i\asymp N^{-1}.
\end{equation}

Write \(s=m+\theta\), where \(m\in\mathbb N_0\) and
\(0\leq\theta<1\).  Applying the scaled Sobolev embedding
\(H^s(B_i)\hookrightarrow C(B_i)\) on the disjoint balls
\(B_i:=B(x_i,q/2)\) gives
\begin{align*}
 q^d|u(x_i)|^2&\lesssim\norm u_{L^2(B_i)}^2
 +\sum_{j=1}^m q^{2j}\sum_{|\alpha|=j}
   \norm{D^\alpha u}_{L^2(B_i)}^2\\
 &\quad+q^{2s}
   \sum_{|\alpha|=m}[D^\alpha u]_{H^\theta(B_i)}^2.
\end{align*}
Here and below, the last term is omitted when \(\theta=0\).
After summation, the local fractional seminorms are bounded by their global
counterparts.  The Fourier inequality
\(q^{2j}|k|^{2j}\lesssim1+q^{2s}|k|^{2s}\), together with
\(N^{-1}\lesssim q^d\), proves \eqref{eq:trace}.

For the MZ inequality, let \(p\in V_K\), define the Voronoi extension
\(\mathcal E_Xp\) to equal \(p(x_i)\) on \(V_i\), and fix an integer
\(s_0>d/2\).  The point-anchored scaled Sobolev
inequality on \(B(x_i,2h)\), obtained from the unit-ball estimate by
dilation, gives
\begin{equation*}
 \int_{V_i}|p(x)-p(x_i)|^2\dd x
 \lesssim\sum_{j=1}^{s_0}h^{2j}
 |p|_{H^j(B(x_i,2h))}^2.
\end{equation*}
Summing this display and using the overlap bound above yields
\begin{equation}\label{eq:sm-MZ-anchored}
 \bignorm{p-\mathcal E_Xp}_2^2
 \lesssim
 \sum_{j=1}^{s_0}h^{2j}|p|_{H^j}^2.
\end{equation}
Periodic Bernstein inequalities and \(Kh\leq1\) therefore imply
\begin{equation*}
 \bignorm{p-\mathcal E_Xp}_2
 \lesssim Kh\norm p_2.
\end{equation*}
Choose \(0<c_{\mathrm{MZ}}\leq1\) so that the implicit constant times
\(c_{\mathrm{MZ}}\) is at most
\(1/2\).  If \(Kh\leq c_{\mathrm{MZ}}\), the triangle and reverse triangle
inequalities then give
\begin{equation*}
 \frac14\norm p_2^2
 \leq \norm{\mathcal E_Xp}_2^2
 =\sum_iw_i|p(x_i)|^2
 \leq\frac94\norm p_2^2.
\end{equation*}
Finally, \eqref{eq:sm-voronoi-weights} implies
\begin{equation*}
 \frac1N\sum_i|p(x_i)|^2\asymp\norm p_2^2,
\end{equation*}
which is precisely the equal-weight estimate \eqref{eq:mz}.
\end{proof}

Henceforth, choose \(h_0>0\) small enough that
\eqref{eq:mz} and \eqref{eq:trace} hold for each Sobolev index used below;
it may be decreased without further comment.

\begin{proof}[Proof of \Cref{lem:reverse-gram-main}]
\emph{Reverse sampling.}
If \(\rho\leq h\), then \(\delta_{h,\rho}\geq2^{-\tau}\), and the
claim follows from \(\norm u_2\leq\norm u_{\tau,\rho}\).  Suppose
\(\rho>h\).  After decreasing \(h_0\), set
\(K=\lfloor c_{\mathrm{MZ}}/(2h)\rfloor\), so that
\(K\geq c_{\mathrm{MZ}}/(4h)\) and \(Kh\leq c_{\mathrm{MZ}}\).
Let \(P_K\) be the Fourier-coordinate projection onto \(V_K\), and write
\(u=p+w\), with \(p=P_Ku\).  The lower MZ bound gives
\begin{equation}\label{eq:sm-low-control}
 \norm p_2^2\lesssim\norm{u|_{X_N}}_N^2
 +\norm{w|_{X_N}}_N^2.
\end{equation}
Every frequency in \(w\) has modulus \(\gtrsim h^{-1}\).  Hence
\begin{equation*}
 \norm w_2^2+h^{2\tau}\norm w_{H^\tau}^2
 \lesssim(h/\rho)^{2\tau}\norm u_{\tau,\rho}^2.
\end{equation*}
Apply \eqref{eq:trace} with \(s=\tau\) to \(w\), substitute in
\eqref{eq:sm-low-control}, and use \(p\perp w\).  This proves
\eqref{eq:reverse-sampling-main}, since
\(\delta_{h,\rho}\asymp(h/\rho)^{2\tau}\) in the present case.

\emph{Gram stability.}
After decreasing \(h_0\), the separation radius is below a fixed fraction
of the injectivity radius.  We first prove the intrinsic estimate
\begin{equation}\label{eq:sm-q-gram}
 \lambda_{\min}\!\left(\frac{\bar{\boldsymbol K}_{\rho,X}}N\right)
 \gtrsim[1+(\rho/q)^2]^{-\tau}.
\end{equation}
Choose a nonzero \(\chi\in C_c^\infty(B(0,1/8))\), put
\(\widetilde\chi(x)=\overline{\chi(-x)}\), and let
\(\phi=\chi*\widetilde\chi\).
Then \(\phi(0)>0\), \(\operatorname{supp}\phi\subset B(0,1/4)\), and
\begin{equation*}
 0\leq\phi^{\ft}(\xi)\lesssim
 (1+(2\pi|\xi|)^2)^{-\tau}.
\end{equation*}
Set \(s=q\wedge\rho\) and define the periodic bump
\begin{equation*}
 \Psi_s(x):=\sum_{n\in\Z^d}\phi\bigl((x+n)/s\bigr).
\end{equation*}
Poisson summation gives
\begin{equation}\label{eq:sm-periodized-bump}
 \Psi_s(x)=s^d\sum_{k\in\Z^d}
 \phi^{\ft}(sk)e^{2\pi i k\cdot x}.
\end{equation}
For \(\alpha\in\mathbb C^N\), write
\(\mathcal A_k=\sum_j\alpha_je^{-2\pi i k\cdot x_j}\).  Since
\(\operatorname{supp}\Psi_s\subset B(0,s/4)\), while distinct sites have
periodic distance at least \(2q\geq2s\), we have
\(\Psi_s(x_i-x_j)=0\) for \(i\ne j\) and \(\Psi_s(0)=\phi(0)\).
Applying \eqref{eq:sm-periodized-bump} to the
resulting quadratic form gives
\begin{equation}\label{eq:sm-ingham-identity}
 \sum_{k\in\Z^d}\phi^{\ft}(sk)|\mathcal A_k|^2
 =s^{-d}\phi(0)\norm\alpha_{\ell^2}^2.
\end{equation}
Since \(s\leq\rho\),
\begin{equation*}
 \bigl(1+(2\pi\rho|k|)^2\bigr)^{-\tau}
 \geq (s/\rho)^{2\tau}
 \bigl(1+(2\pi s|k|)^2\bigr)^{-\tau}
 \gtrsim (s/\rho)^{2\tau}\phi^{\ft}(sk).
\end{equation*}
Consequently, \eqref{eq:sm-ingham-identity} gives
\begin{equation}\label{eq:sm-gram-unnormalized}
 \alpha^*\bar{\boldsymbol K}_{\rho,X}\alpha
 =\sum_k\bar\mu_{\rho,k}|\mathcal A_k|^2
 \gtrsim\rho^{-2\tau}s^{2\tau-d}
 \norm\alpha_{\ell^2}^2.
\end{equation}
The disjoint \(q\)-balls give \(Nq^d\lesssim1\), and hence
\(N^{-1}\gtrsim q^d\).  If \(q\leq\rho\), then \(s=q\) and division
of \eqref{eq:sm-gram-unnormalized} by \(N\) gives
\begin{equation*}
 \lambda_{\min}(\bar{\boldsymbol K}_{\rho,X}/N)
 \gtrsim(q/\rho)^{2\tau}.
\end{equation*}
If \(q>\rho\), then \(s=\rho\), and the same calculation gives
\(\lambda_{\min}(\bar{\boldsymbol K}_{\rho,X}/N)\gtrsim1\).  These two
cases prove \eqref{eq:sm-q-gram}.

Finally, \(q\geq h/\bar\gamma\) implies
\begin{equation}\label{eq:sm-gram-mesh-conversion}
 [1+(\rho/q)^2]^{-\tau}
 \geq\bar\gamma^{-2\tau}[1+(\rho/h)^2]^{-\tau}.
\end{equation}
Since \(G_X=\bar{\boldsymbol K}_{\rho,X}/N\), the fixed factor
\(\bar\gamma^{-2\tau}\) is absorbed into the implicit constant, proving
\eqref{eq:gram-main}.
\end{proof}

\begin{proof}[Proof of \Cref{cor:recovery-stability-main}]
The reverse sampling inequality is the quadratic-form estimate, in the
Loewner order,
\begin{equation}\label{eq:sm-operator-sampling}
 J^*J\lesssim\delta_{h,\rho}I+C_X.
\end{equation}
By \eqref{eq:gram-main}, there is \(c_{\mathrm G}>0\) such that every
eigenvalue \(t\) of \(C_X\) on \(\overline{\operatorname{ran}S^*}\)
satisfies \(t\geq c_{\mathrm G}\delta_{h,\rho}\).
Spectral calculus on \(\overline{\operatorname{ran}S^*}\) gives
\begin{equation*}
 \norm{Q_\eta z}_2^2
 \lesssim\sup_{t\geq c_{\mathrm G}\delta_{h,\rho}}
 \frac{t(\delta_{h,\rho}+t)}{(t+\eta)^2}\norm z_N^2
 \lesssim\norm z_N^2.
\end{equation*}
This proves \eqref{eq:data-stability-main}; the Moore--Penrose calculation is
identical at \(\eta=0\).

Since \(p=\bar T_{\rho,\tau}v=JJ^*v\), for \(\eta>0\) the resolvent
identity gives
\begin{equation*}
 p-Q_\eta(p|_{X_N})
 =J\eta(C_X+\eta I)^{-1}J^*v.
\end{equation*}
Put \(D_\eta:=\eta(C_X+\eta I)^{-1}\).  Factoring this positive
resolvent on both sides of \eqref{eq:sm-operator-sampling} gives, again in
the Loewner order,
\begin{equation*}
\begin{aligned}
 D_\eta J^*JD_\eta
 &\lesssim D_\eta(\delta_{h,\rho}I+C_X)D_\eta\\
 &\lesssim\max\{\delta_{h,\rho},\eta\}D_\eta.
\end{aligned}
\end{equation*}
Indeed, the second inequality follows from
\(
 \eta(\delta_{h,\rho}+t)/(\eta+t)
 \leq\max\{\delta_{h,\rho},\eta\}
\) for \(t\geq0\).  Consequently,
\begin{equation*}
\begin{aligned}
 \norm{JD_\eta J^*v}_2^2
 &\lesssim\max\{\delta_{h,\rho},\eta\}
   \left\langle v,JD_\eta J^*v\right\rangle_{L^2}\\
 &\lesssim\max\{\delta_{h,\rho},\eta\}
   \norm v_2\norm{JD_\eta J^*v}_2.
\end{aligned}
\end{equation*}
This proves \eqref{eq:source-resolvent-main}.  At \(\eta=0\), the
Moore--Penrose identity gives
\[
 p-Q_0(p|_{X_N})
 =J(I-C_X^+C_X)J^*v
 =J\Pi_{\ker S}J^*v.
\]
Here \(\Pi_{\ker S}\) is the RKHS-orthogonal projection onto \(\ker S\),
and \(I-C_X^+C_X=\Pi_{\ker C_X}=\Pi_{\ker S}\).  Restricting
\eqref{eq:sm-operator-sampling} to \(\ker S\) gives the same bound with
\(\delta_{h,\rho}\).
When \(h\leq\rho\),
\(\delta_{h,\rho}\lesssim(h/\rho)^{2\tau}\), and
\(\eta\leq(r/\rho)^{2\tau}\).
Substitution into \eqref{eq:source-resolvent-main} proves the final assertion
of \cref{cor:recovery-stability-main}.
\end{proof}

\subsection{Variance upper bound}
\label{sec:sm-variance-upper}

\begin{proof}[Proof of \Cref{lem:variance-upper-main}]
Let \(t_1,\ldots,t_N\) be the positive eigenvalues of \(C_X\).  On
\(\operatorname{ran}C_X\), \eqref{eq:reverse-sampling-main} and
\eqref{eq:gram-main} imply \(J^*J\lesssim C_X\) in the Loewner order.
Indeed, the first estimate gives \(J^*J\lesssim\delta_{h,\rho}I+C_X\),
while the second gives \(\delta_{h,\rho}I\lesssim C_X\) on this range.
Since the covariance of
\(\xi\) in \(E_N\) is \((\sigma^2/N)I\), it follows that
\begin{equation}\label{eq:sm-variance-effective-dimension}
 \E\norm{Q_\eta\xi}_2^2
 =\frac{\sigma^2}{N}\norm{Q_\eta}_{\mathrm{HS}}^2
 \lesssim\frac{\sigma^2}{N}
 \sum_{j=1}^N\left(\frac{t_j}{t_j+\eta}\right)^2,
\end{equation}
where the quotient equals one when \(\eta=0\).

It remains to bound the number of directions effectively transmitted by
ridge.  Assume first that \(\eta>0\), put
\(r_\eta:=\rho\eta^{1/(2\tau)}\), and take
\(K=\lceil r_\eta^{-1}\rceil\).  The hypothesis \(r\leq1\) gives
\(r_\eta\leq1\), and hence \(K\asymp r_\eta^{-1}\).  Let \(P_K\) be the
Fourier-coordinate projection onto \(V_K\), which is orthogonal in both
\(L^2\) and the normalized Mat\'ern RKHS, and put
\[
 A_\eta:=\bigl[C_X(C_X+\eta I)^{-1}\bigr]^2.
\]
Its trace is \(\sum_j(t_j/(t_j+\eta))^2\).  Moreover,
\((t/(t+\eta))^2\leq\min\{1,t/(4\eta)\}\) implies
\(0\preceq A_\eta\preceq I\) and
\(A_\eta\preceq C_X/(4\eta)\).  Thus, with
\(P_K^\perp=I-P_K\),
\begin{equation*}
\begin{aligned}
 \tr A_\eta
 &=\tr(P_KA_\eta P_K)+\tr(P_K^\perp A_\eta P_K^\perp)\\
 &\leq\operatorname{rank}(P_K)
   +\frac1{4\eta}\tr(P_K^\perp C_XP_K^\perp).
\end{aligned}
\end{equation*}
To evaluate the discarded trace, set
\(\psi_k:=\bar\mu_{\rho,k}^{1/2}e_k\).  The \(\psi_k\) form an
orthonormal basis of the normalized Mat\'ern RKHS, and
\(\norm{S\psi_k}_N^2=\bar\mu_{\rho,k}\).  Therefore
\begin{equation*}
 \tr(P_K^\perp C_XP_K^\perp)
 =\sum_{\norm k_\infty>K}\bar\mu_{\rho,k}
 \lesssim\rho^{-2\tau}K^{d-2\tau},
\end{equation*}
where the last bound uses \(2\tau>d\).  Since
\(\operatorname{rank}(P_K)\lesssim K^d\), the preceding displays and the
choice of \(K\) give
\[
 \sum_{j=1}^N\left(\frac{t_j}{t_j+\eta}\right)^2
 \lesssim r_\eta^{-d}.
\]
The sum is also at most \(N\lesssim h^{-d}\), so
\(r=\max\{h,r_\eta\}\) gives the sharper bound \(\lesssim r^{-d}\).
When \(\eta=0\), the same conclusion follows directly from
\(N\lesssim h^{-d}=r^{-d}\).  Substitution into
\eqref{eq:sm-variance-effective-dimension} proves
\eqref{eq:variance-upper-main}.

The oracle proof uses this lemma only on the branch \(r\leq1\).  Larger
ridge scales give a squared-bias term of order \(A^2\) through
\eqref{eq:ridge-all-scale-main} and are covered by the \(A^2\) cap.
\end{proof}

\section{Finite-exclusion Fourier frame}\label{sec:sm-frame}

This section supplies the annular-packing, Poisson-summation, and
Gershgorin calculation behind the frame estimate in the main paper.

\begin{proof}[Proof of \Cref{lem:frame}]
Put \(m:=\max\{1,|F|\}\) and, for \(\alpha\in\mathbb C^N\), write
\begin{equation*}
 \mathcal A_k:=\sum_{j=1}^N\alpha_je^{-2\pi i k\cdot x_j}.
\end{equation*}
Choose a fixed even cutoff \(w\in C_c^\infty(\mathbb R^d)\) such that
\(0\leq w\leq1\), \(w\) is positive near the origin, and
\(\operatorname{supp}w\subset B(0,R_w)\).  Form
\begin{equation*}
 \Theta_M(x)=\sum_{k\in\Z^d}w(k/M)e^{2\pi i k\cdot x}.
\end{equation*}
Poisson summation and rapid decay of \(w^{\ft}\) give, for every
\(p>d\),
\begin{equation}\label{eq:sm-frame-localization}
 |\Theta_M(x)|\leq C_pM^d
 (1+M\dist(x,0))^{-p},
 \qquad x\in\T^d.
\end{equation}
Moreover, lattice-point counting on a ball where \(w\) is bounded below
shows that, for \(M\geq M_w\),
\begin{equation}\label{eq:sm-frame-diagonal}
 \Theta_M(0)=\sum_kw(k/M)\geq c_wM^d.
\end{equation}

For fixed \(i\), put
\begin{equation*}
 \mathcal R_s(i):=\{j\ne i:2^sq\leq\dist(x_i,x_j)<2^{s+1}q\},
 \qquad s\geq0.
\end{equation*}
The disjoint \(q\)-balls centered at the sites give
\begin{equation}\label{eq:sm-frame-annular-count}
 |\mathcal R_s(i)|\leq C_d2^{sd}.
\end{equation}
Enlarge \(C_d\), if necessary, so that also \(Nq^d\leq C_d\).
Take \(M=L_0/q\), where \(L_0\geq1\) is fixed below.  Combining
\eqref{eq:sm-frame-localization} and
\eqref{eq:sm-frame-annular-count} bounds the off-diagonal row sum by
\(C_{d,p}M^d\sum_{s\geq0}2^{sd}(1+L_02^s)^{-p}\).  Since \(p>d\),
\begin{equation}\label{eq:sm-frame-row-sum}
 \sum_{j\ne i}|\Theta_M(x_i-x_j)|
 \leq C'_{d,p}L_0^{-p}M^d.
\end{equation}
Choose \(L_0\) large enough that this row sum is at most \((c_w/2)M^d\) and
\begin{equation}\label{eq:sm-frame-deletion-choice}
 mC_dq^{-d}\leq(c_w/4)M^d
 =(c_w/4)L_0^dq^{-d},
\end{equation}
where the same \(C_d\) controls both the annular and global packing bounds.
Taking \(h_F\) small enough ensures \(M\geq M_w\).

The matrix \([\Theta_M(x_i-x_j)]_{i,j}\) is Hermitian.  By
\eqref{eq:sm-frame-diagonal} and \eqref{eq:sm-frame-row-sum},
Gershgorin's theorem bounds its smallest eigenvalue below by
\((c_w/2)M^d\).  Its quadratic form satisfies
\begin{equation*}
 \sum_{i,j}\overline\alpha_i\alpha_j\Theta_M(x_i-x_j)
 =\sum_{k\in\Z^d}w(k/M)|\mathcal A_k|^2
 \leq\sum_{|k|\leq R_wM}|\mathcal A_k|^2.
\end{equation*}
It follows that
\begin{equation}\label{eq:sm-frame-window}
 \sum_{|k|\leq R_wM}|\mathcal A_k|^2
 \geq(c_w/2)M^d\norm\alpha_{\ell^2}^2.
\end{equation}
For every \(k\), Cauchy--Schwarz and packing give
\begin{equation*}
 |\mathcal A_k|^2\leq N\norm\alpha_{\ell^2}^2,
 \qquad Nq^d\leq C_d.
\end{equation*}
Thus \eqref{eq:sm-frame-deletion-choice} allows the frequencies in
\(F\cap\{|k|\leq R_wM\}\) to be subtracted from
\eqref{eq:sm-frame-window}, proving
\eqref{eq:frame} with \(L_F=R_wL_0\).
\end{proof}

\section{Fixed finite Fourier mean spaces}\label{sec:sm-augmentation}

\Cref{prop:target-scale-mode} rules out a missing-intercept explanation using
a target-scale mode.  The following strengthening shows that adjoining any
fixed finite Fourier mean space still leaves the kernel-resolution barrier.
We use complex spans for bookkeeping; a fixed real trigonometric space is
contained in its symmetric complexification.  By the representer theorem, the
corresponding Mat\'ern interpolants and KRR fits lie in the augmented span
below, so its approximation lower bound applies to them.

\begin{proposition}[Fixed finite mean spaces preserve the kernel-resolution barrier]
\label{prop:sm-augmentation}
Fix a finite \(\Lambda\subset\Z^d\), independently of
\(N,h,\ell,\rho\), and \(X_N\), and put
\begin{equation*}
 \mathcal P_\Lambda:=\operatorname{span}_{\mathbb C}\{e_k:k\in\Lambda\},
 \qquad
 \cV_{\rho,X_N}:=\operatorname{span}_{\mathbb C}
 \{\bar K_{\rho,\tau}(\,\cdot-x_j):1\leq j\leq N\}.
\end{equation*}
There exist a frequency \(\nu_\Lambda\notin\Lambda\) and constants
\(c_\Lambda,h_\Lambda>0\), depending only on
\(d,\beta,\tau,\bar\gamma,\Lambda\), such that, for every
\(X_N\) with \(h\leq h_\Lambda\) and \(\gamma\leq\bar\gamma\), and every
\(0<\rho\leq1\),
\begin{equation}\label{eq:sm-augmented-span}
 \dist_{L^2}^2\left(e_{\nu_\Lambda},
 \mathcal P_\Lambda+\cV_{\rho,X_N}\right)
 \geq c_\Lambda[1+(\rho/h)^2]^{-2\tau}.
\end{equation}
Moreover, for every such \(X_N,\rho\), and \(0<\ell\leq1\), there is a real-valued
\(f_\Lambda\in\cF_\beta(A,\ell)\) such that
\begin{equation}\label{eq:sm-augmented-class}
 \inf_{g\in\mathcal P_\Lambda+\cV_{\rho,X_N}}
 \bignorm{f_\Lambda-g}_2^2
 \geq c_\Lambda A^2[1+(\rho/h)^2]^{-2\tau}.
\end{equation}
If \(h\leq\rho\), the right-hand side is bounded below by
\(c_\Lambda A^2(h/\rho)^{4\tau}\).
\end{proposition}

\begin{proof}
Choose \(\nu=\nu_\Lambda\notin\Lambda\) with minimal modulus, breaking ties
arbitrarily.  Then \(|\nu|\) is bounded in terms of \(d\) and
\(|\Lambda|\).
For \(p\in\mathcal P_\Lambda\) and
\(g_\alpha=\sum_j\alpha_j\bar K_{\rho,\tau}(\cdot-x_j)\), put
\(F=\Lambda\cup\{\nu\}\) and
\(\mathcal A_k=\sum_j\alpha_je^{-2\pi i k\cdot x_j}\).
After decreasing \(h_\Lambda\) if needed, apply \cref{lem:frame} with this
set \(F\).  On its frequency window,
\(\bar\mu_{\rho,k}^2\gtrsim[1+(\rho/q)^2]^{-2\tau}\).  Hence, after
enlarging the resulting nonnegative sum,
\begin{equation*}
 \sum_{k\notin F}\bar\mu_{\rho,k}^2|\mathcal A_k|^2
 \geq c_\Lambda[1+(\rho/q)^2]^{-2\tau}
 q^{-d}\norm\alpha_{\ell^2}^2
 \geq D_\Lambda|\mathcal A_\nu|^2,
\end{equation*}
where the second inequality uses \(q\geq h/\bar\gamma\),
\(|\mathcal A_\nu|^2\leq N\norm\alpha_{\ell^2}^2\), and
\(Nq^d\leq C_d\), with
\(D_\Lambda:=c_\Lambda[1+(\rho/h)^2]^{-2\tau}\leq1\), after decreasing
\(c_\Lambda\) if necessary.
Because \(p\) has Fourier support in \(\Lambda\), Parseval bounds
\(\bignorm{e_\nu-p-g_\alpha}_2^2\) below by
\[
 |1-\bar\mu_{\rho,\nu}\mathcal A_\nu|^2
 +D_\Lambda|\mathcal A_\nu|^2.
\]
Since \(0<\bar\mu_{\rho,\nu}\leq1\) and \(D_\Lambda\leq1\),
\begin{equation*}
 \inf_{a\in\mathbb C}
 \left\{|1-\bar\mu_{\rho,\nu}a|^2+D_\Lambda|a|^2\right\}
 =\frac{D_\Lambda}{\bar\mu_{\rho,\nu}^2+D_\Lambda}
 \geq\frac{D_\Lambda}{2}.
\end{equation*}
The preceding expression is therefore at least \(D_\Lambda/2\), proving
\eqref{eq:sm-augmented-span}.

Let \(W=\mathcal P_\Lambda+\cV_{\rho,X_N}\).  If \(\nu=0\), set
\(\psi_\nu=e_0\).  Otherwise, let
\(c_\nu=(e_\nu+e_{-\nu})/\sqrt2\) and
\(s_\nu=(e_\nu-e_{-\nu})/(\sqrt2\,i)\).
Because \(e_\nu=(c_\nu+i s_\nu)/\sqrt2\) and \(W\) is complex linear,
\eqref{eq:sm-augmented-span} implies that at least one
\(\psi_\nu\in\{c_\nu,s_\nu\}\) satisfies
\begin{equation*}
 \dist_{L^2}^2(\psi_\nu,W)
 \geq\tfrac12\dist_{L^2}^2(e_\nu,W).
\end{equation*}
Set
\(f_\Lambda:=A[1+(2\pi\ell|\nu|)^2]^{-\beta/2}\psi_\nu\).
Then \(f_\Lambda\in\cF_\beta(A,\ell)\), and the scalar prefactor is bounded
below by \(c_\Lambda A\) because \(\ell\leq1\) and \(\nu\) is fixed.
This proves \eqref{eq:sm-augmented-class}.
\end{proof}

The fixedness assumption is essential: a mean space whose dimension is of
order \(h^{-d}\) can contain every geometrically resolved Fourier mode.

\section{Two-dimensional information exponents}
\label{sec:sm-dimension-numerics}

Using the finite-section construction of \cref{sec:numerics}, we perform a
focused check of the dimension dependence in the two statistical terms.  For \(d=2\)
and \(\beta=\tau=2\), the target- and kernel-information powers are
\(2\beta/(2\beta+d)=2/3\) and \(4\tau/(4\tau+d)=4/5\), rather than
\(4/5\) and \(8/9\) as in \(d=1\).  We use an \(8\times8\) periodic
lattice, fix \(\ell/h=8\), normalize \(A=1\), set
\(\sigma^2=N\ell^2/\Iell\), and vary \(\Iell=2^j\), \(4\leq j\leq16\).
The target and kernel Fourier sections are symmetric boxes with
coordinate cutoffs
\begin{equation*}
 K_T:=\max\{8,\lceil5/\ell\rceil\}=8,
 \qquad
 K:=\max\{K_T,64,\lceil8/\rho\rceil\}=64.
\end{equation*}
Conjugation symmetry gives the same largest bias eigenvalue as the
corresponding real sine--cosine section.  The ridge grid, zero-estimator
endpoint, and reliability criterion are those of \cref{sec:numerics}.
The choice \(\rho/\ell=5/4\) isolates the target-information branch, while
\(\rho/\ell=1/4\) exposes the kernel-information branch before the geometric
term becomes active.  The corresponding curves in
\cref{fig:sm-dimension-numerics} follow the prescribed slopes \(-2/3\) and
\(-4/5\) over their information-dominated ranges.  Increasing \(K\) from
\(64\) to \(96\) changes the grid-refined oracle risks by at most
\(0.004\%\); independently replacing \(5/\ell\) by \(7.5/\ell\) in
\(K_T\) changes them by at most \(0.18\%\).  This calculation checks the
dimension dependence of the information powers; it is not an additional
design-robustness experiment.

\begin{figure}[!htbp]
\centering
\includegraphics[width=.80\textwidth]{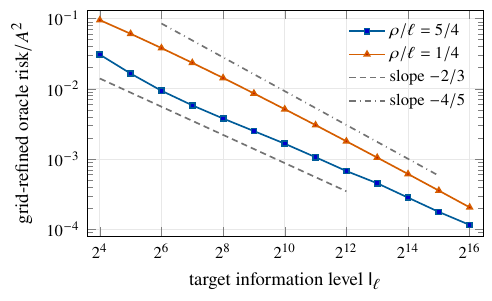}
\caption{Focused \(d=2\) information-level experiment on an \(8\times8\)
periodic lattice with \(\beta=\tau=2\) and \(\ell/h=8\).  The curves for
\(\rho/\ell=5/4\) and \(\rho/\ell=1/4\) exhibit the target- and
kernel-information slopes, respectively.  The dashed and dash-dotted lines
are prescribed-slope guides, vertically offset for legibility.}
\label{fig:sm-dimension-numerics}
\end{figure}

\end{document}